\documentclass[final]{article}

\usepackage{graphicx}
\usepackage{bbm,amssymb,color}
\usepackage{amsmath}
\usepackage{subcaption}
\usepackage{wrapfig}
\usepackage{bm}
\usepackage[margin=1in]{geometry}
\usepackage{todonotes}
\usepackage{hyperref}
\hypersetup{
	pdfauthor={},
	pdftitle={},
	pdfkeywords={},
	pdfsubject={},
	pdfcreator={Emacs 25.1.1 (Org mode 8.3.4)}, 
	pdflang={English},
	final
}

\newcommand{\Z}{\ensuremath{\mathbbm{Z}}}
\newcommand{\R}{\ensuremath{\mathbbm{R}}}
\newcommand{\C}{\ensuremath{\mathbbm{C}}}

\newcommand{\ci}{\ensuremath{\mathrm{i}}}

\renewcommand{\Re}{\operatorname{Re}}
\renewcommand{\Im}{\operatorname{Im}}

\DeclareMathOperator{\tr}{tr}
\DeclareMathOperator{\Span}{span}
\DeclareMathOperator{\adj}{adj}

\newtheorem{theorem}{Theorem}[section]
\newtheorem{proposition}[theorem]{Proposition}
\newtheorem{lemma}[theorem]{Lemma}
\newtheorem{definition}[theorem]{Definition}
\newtheorem{corollary}[theorem]{Corollary}
\newtheorem{example}[theorem]{Example}
\newtheorem{remark}[theorem]{Remark}

\newenvironment{proof}{\par\noindent{\em proof.} 
\normalfont}{\qed\par}
\DeclareRobustCommand{\qed}{%
  \ifmmode 
  \else \leavevmode\unskip\penalty9999 \hbox{}\nobreak\hfill
  \fi
  \quad\hbox{\qedsymbol}}
\newcommand{\openbox}{\leavevmode
  \hbox to.77778em{%
  \hfil\vrule
  \vbox to.675em{\hrule width.6em\vfil\hrule}%
  \vrule\hfil}}
\newcommand{\qedsymbol}{\openbox}

\title{Skew parallelogram nets and universal factorization}
\date{}
\author{Tim Hoffmann
\and Andrew O. Sageman-Furnas
\and Jannik Steinmeier
}

\begin{document}

\maketitle

\begin{abstract}
We obtain many objects of discrete differential geometry as reductions of skew parallelogram nets, a system of lattice equations that may be formulated for any unit associative algebra. The Lax representation is linear in the spectral parameter, and paths in the lattice give rise to polynomial dependencies. We prove that generic polynomials in complex $2\times2$ matrices factorize, implying that skew parallelogram nets encompass all systems with such a polynomial representation.

We demonstrate factorization in the context of discrete curves by constructing pairs of Bäcklund transformations that induce Euclidean motions on discrete elastic rods. More generally, we define a hierarchy of discrete curves by requiring such an invariance after an integer number of Bäcklund transformations.

Moreover, we provide the factorization explicitly for discrete constant curvature surfaces and reveal that they are slices in certain 4D cross-ratio systems. 
Encompassing the discrete DPW method, this interpretation constructs such surfaces from given discrete holomorphic maps.
\end{abstract}

\tableofcontents

\newpage

\section{Introduction}
\label{sec:intro}

We propose an equation system in unit associative algebras which encompasses many geometric structures. For example, all special classes of discrete surfaces from~\cite{bobenko2008discrete}, integrable systems such as the cross-ratio equation or polygon recutting and even motion polynomials in linkage kinematics can be found as special cases or reductions of this system. 
From the perspective of discrete differential geometry we discuss general properties and some applications of this equation system.

We consider maps $p:\{\text{directed edges of }\mathbb{Z}^n\}\to\mathcal{A}$ into a unit associative algebra $\mathcal{A}$ for which the values $p,q,r,s$ at the four edges of every elementary quad fulfill the following two simple equations (Definition \ref{def:parnets})
\begin{align*}
p+q=r+s,\qquad qp=sr.
\end{align*}

We call such maps \emph{skew parallelogram nets} since for the algebra of quaternions this equation system describes parallelograms folded along a diagonal, in other words, skew parallelograms.

These folded parallelograms are central to several variants of polygon dynamics such as the bicycle transformation~\cite{tabachnikov2012discrete}, periodic Darboux transformations \cite{cho2023periodic} or cross-ratio dynamics~\cite{arnold2022cross,affolter2023integrable}. These are based on cross-ratio systems \cite{nijhoff1997some} which we describe as skew parallelogram nets. 
We also find a connection to recent work~\cite{izosimov2023recutting} on polygon recutting~\cite{adler1993recuttings} where the dynamic is described in terms of refactorization of polynomials.

Factorization of such polynomials is extensively studied in the field of linkage kinematics \cite{li2019factorization} where it is used to realize rational motions with closed linkages. This has been studied for Euclidean motions \cite{hegedus2013factorization}, hyperbolic motions \cite{scharler2020quadratic,scharler2021algorithm} and, recently, conformal motions \cite{li2023geometric}.

In the context of discrete differential geometry skew parallelogram nets have been studied in \cite{schief2007chebyshev} for quaternions and have been suggested in~\cite{bobenko2008discrete} for associative algebras. In section \ref{sec:parnets}, we expand on their ideas and provide an evolution, a linear Lax representation and an associated family. We briefly discuss 3D-consistency and Bäcklund transformations. 
We investigate skew parallelogram nets for the algebra $\mathbb{C}^{2\times2}$ and its subalgebra isomorphic to the quaternions. We solve the initial value problem generically using refactorization of polynomials. Through factorization one can also interpret more complex polynomial systems in terms of skew parallelogram nets. Sections \ref{sec:elastica} and \ref{sec:surfaces} can be seen as a demonstration this concept.

Throughout the paper, skew parallelogram nets over quaternions are of special interest since they are known to encompass discrete Lund-Regge systems \cite{schief2007chebyshev}, asymptotic K-nets (also called pseudospherical surfaces)~\cite{discretizationOfSurfacesIntegrable}, constant mean curvature nets \cite{hoffmann16} and, via Bäcklund transformations of curves, a discrete Hashimoto flow \cite{pinkallSmokeRingFlow,hoffmannSmokeRingFlow}.

In section \ref{sec:elastica} we study curves in Euclidean space that stay invariant under a sequence of such Bäcklund transformations. In~\cite{hoffmannSmokeRingFlow} it is shown that curves invariant under two Bäcklund transformation are discrete elastic rods. We apply factorization to prove the converse. Discrete elastic rods are known to be invariant under the smooth Hashimoto flow \cite{lagrangeTop} and, thus, curves invariant under the smooth flow agree with curves invariant under the discrete flow.  Another connection between Bäcklund transformations and elastic rods has recently been observed in \cite{bor2023bicycling}. For more context on (discrete) elastic rods and their various applications see \cite{bergou2008elasticrods} and references therein. 
Moreover, we provide an algorithm for the construction of invariant curves for any number of Bäcklund transformations. This approach might give a viable way towards a discrete version of the Hierarchy of commuting Hamiltonian flows given in \cite{langer1991poisson} which also is studied in space forms \cite{chern2018commuting} or centro-affine space \cite{calini2013integrable}. 

In section \ref{sec:surfaces} we provide a description for cross-ratio systems in terms of skew parallelogram nets and show how they give rise to different types of constant curvature surfaces. In particular, we find asymptotic K-nets as a natural reduction while more complex surfaces like constant mean curvature nets~\cite{discretizationOfSurfacesIntegrable} and circular K-nets~\cite{cKnets} arise as slices in a 4D lattice. We investigate two reductions of 4D cross-ratio systems which correspond to these surfaces and show how to recover them from 2D cross-ratio systems which discretize holomorphic data. This method appears to be a discrete DPW method~\cite{dorfmeister1998weierstrass} similar to the one presented in~\cite{ogata2017construction,discreteDPW}. As an application of this method, we construct a lattice of breather transformations of a surface.

In section \ref{sec:tnets} we conclude the paper with an approach on how to study our system in more general algebras. For the algebra $\mathbb{H}^{2\times2}$ this is already known in the context of isothermic nets and their transformation theory~\cite{bobenko1996discrete,doliwa2007generalized,hertrich2000transformations,burstall2015discrete}. Isothermic nets in arbitrary dimension \cite{schief2001isothermic} and in various geometries such as sphere geometries \cite{isothermicInSphereGeometriesMoutard} are all described by Moutard nets in quadrics. We show how skew parallelogram nets are equivalent to Moutard nets in different ways and how this can be used to define an associated family.

How to give the associated family a geometric interpretation is an interesting open question that we will continue to investigate alongside the above mentioned question about the Hamiltonian flows. Equally intriguing is the prospect of finding cluster and Poisson structures for skew parallelogram nets.\\

\textbf{Acknowledgements.} We would like to thank Alexander Bobenko for his valuable ideas that led to the algorithm for $n$-invariant curves and Niklas Affolter and Gudrun Szewieczek for many fruitful discussions. This research is partially funded by the Deutsche Forschungsgemeinschaft (DFG - German Research Foundation) - Project-ID 195170736 - TRR109 "Discretization in Geometry and Dynamics".

\subsection*{Notation}

Throughout, we consider (subsets of) the $n$-dimensional integer lattice $\mathbb{Z}^n$ as a graph. We identify $\mathbb{Z}^n$ with its vertices $\mathcal{V}(\mathbb{Z}^n)$ and write
$\mathcal{E}(\mathbb{Z}^n)$ for its undirected set of edges. We often consider the set $\vec{ \mathcal{E}}(\mathbb{Z}^n)$ of directed edges $(x,x+e_j)$ where $x \in \mathcal{V}(\mathbb{Z}^n)$ and $e_j\in\mathbb{Z}^n$ is the $j$-th unit vector. 
We investigate vertex based maps $f:\mathbb{Z}^n\to \mathcal{A}$ and edge based maps $p:\vec{ \mathcal{E}}(\mathbb{Z}^n)\to \mathcal{A}$, where $\mathcal{A}$ is some algebra. For edge based maps we write $p=(p^1,\cdots,p^n)$ where $p^j(x):=p(x,x+e_j)$ denotes the value of $p$ at the edge heading from $x\in\mathbb{Z}^n$ in direction $j$. For $n\leq3$ we denote $u:=p^1,v:=p^2,w:=p^3$ for simplicity.

\begin{figure}[h!]
  \centering
  \includegraphics[width=0.17\textwidth]{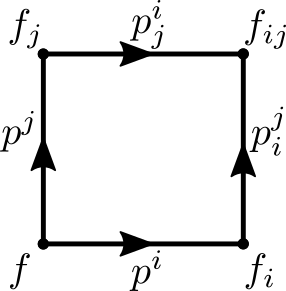}
  \hspace{1.5cm}
  \includegraphics[width=0.23\textwidth]{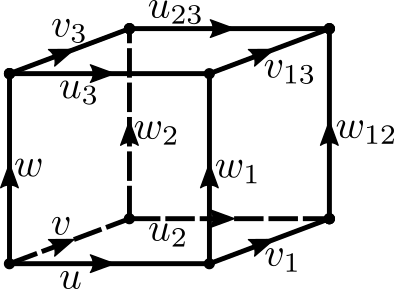}
  \caption{Notation at vertices and edges on a quad and a cube.}
  \label{fig:notation}
\end{figure}

For a vertex based map we use the short notation $f$ for $f(x),x\in\mathbb{Z}^n$ and denote shifts with an index: We write $f_i$ for $f(x+e_i)$ and $f_{\bar{i}}$ for $f(x-e_i)$. Similarly, for an edge based map we write $p^j$ for $p^j(x),x\in\mathbb{Z}^n$ and denote $p^j(x+e_i)$ and $p^j(x-e_i)$ by $p^j_i$ and $p^j_{\bar{i}}$ respectively.
We write $\bm{0}:=(0,\cdots,0)\in\mathbb{Z}^n$.

If $\mathcal{A}$ has a unit we denote it by $1$ and write $\lambda+q$ for $\lambda\cdot 1+q$ for scalars $\lambda$ and $q\in\mathcal{A}$.

\section{Skew parallelogram nets in algebras}
\label{sec:parnets}

\subsection{The skew parallelogram equations and evolution}

Let $\mathcal{A}$ be a unit associative algebra over the field $\mathbb{K}$, with $\mathbb{K}$ being the real or complex numbers.

\begin{definition}
\label{def:parnets}
An \emph{$n$-dimensional skew parallelogram net} is a map $p=(p^1,...,p^n):\vec{ \mathcal{E}}(\mathbb{Z}^n)\to\mathcal{A}$, such that for each quad we have
\begin{align}
\label{eq:pareqadd}
p^{i}+p^{j}_i=p^{j}+p^{i}_j
\end{align}
and
\begin{align}
\label{eq:pareqmult}
p^j_ip^{i}=p^{i}_jp^j.
\end{align}
\end{definition}

For $2$-dimensional skew parallelogram nets the equations read $u+v_1=v+u_2$ and $v_1u=u_2v$. For brevity we refer to $p$ as a \emph{parallelogram net}.

A parallelogram net $p$ defines a vertex based map $f$ by $f_i:=f+p^i$ for a given $f(\bm0)$. We call $f$ the \textit{primitive map} of $p$. Throughout this paper we will relate $f$ to various geometric objects, in particular, to curves and surfaces in $\mathbb{R}^3$. The objects we get depend on the choice of $\mathcal{A}$ and further constraints on the parallelogram net.

\begin{remark}
\label{rem:multPrimitiveMap}
Sometimes, using other primitive maps is useful. For example, in section \ref{sec:tnets}, we will use a vertex based map defined multiplicatively by $f_i :=p^if^{-1}$ for a given $f(\bm0)$.
\end{remark}

One can easily construct a new parallelogram net $\tilde{p}$ from a given solution $p$:
\begin{itemize}
\item By conjugation: $\tilde{p}^j=cp^jc^{-1}$, $c\in\mathcal{A}$ invertible.
\item By scaling: $\tilde{p}^j=sp^j$, $s\in\mathbb{K}\setminus\{0\}$.
\item By adding a scalar: $\tilde{p}^j=p^j+r$, $r\in\mathbb{K}$.
\end{itemize}
If two parallelogram nets can be obtained from each other by these operations, we call them \textit{primarily equivalent}. Both describe the same solution of \eqref{eq:pareqadd},\eqref{eq:pareqmult}, but they might carry different algebraic information.

We call a parallelogram net \textit{evolvable} if its diagonals $p^j-p^i$ for $i\neq j$ are invertible.
\begin{proposition}
An evolvable parallelogram net $p$ satisfies the evolution
\begin{align}
\label{eq:parevol}
p^i_j=(p^j-p^i)p^i(p^j-p^i)^{-1}.
\end{align}
\end{proposition}
\begin{proof}
Plugging \eqref{eq:pareqadd} into \eqref{eq:pareqmult} yields $(p^{j}-p^{i})p^{i}=p^{i}_j(p^j-p^i)$.
\end{proof}

Hence, from its values on the axes of $\mathbb{Z}^n$ we can reconstruct $p$ on the positive quadrant. Similarly, we can compute the negative quadrant from
\begin{align}
\label{eq:parevolbw}
p^i=(p^j_i-p^i_j)^{-1}p^i_j(p^j_i-p^i_j).
\end{align}
However, to reconstruct all quadrants from the values on the axes we need special properties of the the algebra $\mathcal{A}$ as will be explained in section \ref{sec:refac}.

In the simplest case, where $\mathcal{A}$ is commutative, the evolution becomes the identity. We will only consider non-commutative algebras.

The evolution \eqref{eq:parevol} has many invariants. For example, if $\mathcal{A}$ is a matrix algebra, the matrices $p^j$ and $p^j_i$ are similar, implying that their characteristic polynomials coincide. In the case where $\mathcal{A}$ is the algebra of quaternions, the invariants are the real part and the length. In all cases, every invariant $\alpha$ is an edge based map which has the \emph{labelling property} meaning that $\alpha^i$ only depends on the $i$-th coordinate. Equivalently, this means that we have $\alpha^i_j=\alpha^i$ for $i\neq j$.

\begin{remark}
The parallelogram equations and evolution have been studied before. For the algebra of quaternions they appear as discrete flows on curves in \cite{pinkallSmokeRingFlow,hoffmannSmokeRingFlow} and as discrete Lund-Regge system in \cite{schief2007chebyshev}. In particular, in \cite{schief2007chebyshev} we find a detailed discussion on consistency and permutability and also an associated family. In unit associative algebras the system is suggested in \cite[exercise 6.19]{bobenko2008discrete}, where we find a proof that the evolution \eqref{eq:parevol} is 3D consistent. Thus, generically, we can obtain a parallelogram net from any initial values on the coordinate axes.
\end{remark}

\subsection{Multidimensional consistency and Bäcklund transformations}

We will briefly touch the topic of consistency and permutability of parallelogram nets (again, see \cite{schief2007chebyshev,bobenko2008discrete}). We verify that the equations stay preserved throughout a cube.

\begin{proposition}
Consider invertible values $u,v,w\in\mathcal{A}$ on the edges of a cube as in Figure \ref{fig:notation}. If the parallelogram equations \eqref{eq:pareqadd},\eqref{eq:pareqmult} are fulfilled on five faces of a cube, then they are also fulfilled on the sixth face.
\end{proposition}
\begin{proof}
We can obtain the equations on the top face from the equations on the other faces:
\begin{align*}
v_3+u_{23}&=(-w+v+w_2)+(-w_2+u_2+w_{12})\\
&=(-w+u+w_1)+(-w_1+v_1+w_{12})=u_2+v_{12}
\end{align*}
and
\begin{align*}
u_{23}v_3=u_{23}v_3ww^{-1}=u_{23}w_2vw^{-1}=w_{12}u_2vw^{-1}=w_{12}v_1uw^{-1}=v_{12}w_1uw^{-1}=v_{12}u_3.
\end{align*}
Other faces work similarly.
\end{proof}

In integrable discrete differential geometry a two-layered $(n+1)$-dimensional net is seen as two $n$-dimensional nets related by their natural transformation (e.g.~\cite{bobenko2008discrete} and references therein).

\begin{definition}
Two $n$-dimensional parallelogram nets $p$ and $\tilde{p}$ are related by a \emph{Bäcklund transformation} 
if there exists a map $v:\mathbb{Z}^n\to\mathcal{A}$, such that at every edge in $\vec{ \mathcal{E}}(\mathbb{Z}^n)$ we have
\begin{align*}
p^i+v_i=v+\tilde{p}^{i} \qquad\text{and}\qquad v_ip^{i}=\tilde{p}^{i}v.
\end{align*}
\end{definition}

The nets $p$ and $\tilde p$ should be considered to live on different copies of the integer graph $(\mathcal{V}(\mathbb{Z}^n),\vec{ \mathcal{E}}(\mathbb{Z}^n))$. The map $v$ should then be considered to live on the directed edge $(x,\tilde x)$ where $x$ and $\tilde x$ are corresponding points of these two copies. In this way a Bäcklund transformation indeed just becomes a two-layered $(n+1)$-dimensional parallelogram net. 
We will study Bäcklund transformations of curves in section~\ref{sec:elastica}.

\subsection{Lax representation and associated family}

A key property of parallelogram nets is that they possess a linear Lax representation.

\begin{definition}[Lax representation]
To an edge based map $p:\vec{ \mathcal{E}}(\mathbb{Z}^n)\to\mathcal{A}$ we assign $P:\vec{ \mathcal{E}}(\mathbb{Z}^n)\times\mathbb{K}^2\to\mathcal{A}$ defined by
\begin{align}
\label{eq:laxpair}
P^i(\lambda,\mu):=\lambda+\mu p^i.
\end{align}
\end{definition}

If $p$ is a parallelogram net the map $P(\lambda,\mu)$ is primarily equivalent to $p$ for any fixed $\lambda,\mu$ (with $\mu\neq0$). As a function of $\lambda,\mu$ the map $P$ is a Lax representation for parallelogram nets:

\begin{proposition}
An edge based net $p$ is a parallelogram net if and only if the Lax representation \eqref{eq:laxpair} fulfills the compatibility condition $P^j_iP^i=P^i_jP^j$.
\end{proposition}
\begin{proof}
Simply expand
\begin{align*}
P^j_iP^i&=(\lambda+\mu p^j_i)(\lambda+\mu p^i)=\lambda^2+\lambda\mu(p^i+p^j_i)+\mu^2 p^j_ip^i,\\
P^i_jP^j&=(\lambda+\mu p^i_j)(\lambda+\mu p^j)=\lambda^2+\lambda\mu(p^j+p^i_j)+\mu^2 p^i_jp^j.
\end{align*}
Hence, $P^j_iP^i=P^i_jP^j$ for all $\lambda,\mu\in\mathbb{K}$ is equivalent to the parallelogram equations \eqref{eq:pareqadd} and \eqref{eq:pareqmult}.
\end{proof}

This also works with single variable Lax representations like $\lambda+p^i,1+\mu p^i$ or $\lambda+\frac1\lambda p^i$. These representations are all equivalent and it will turn out useful to switch between them. Often, we will treat $\lambda$ and $\mu$ as differentiable functions of a parameter $t\in\mathbb{R}$. The variables $\lambda,\mu,P^i$ and $\Phi$ (defined below) then depend on this parameter. Usually, this will not be stated explicitly. For the functions $\lambda,\mu$ we assume $\lambda'\mu-\lambda\mu'\neq0$ for all $t\in\mathbb{R}$.

\begin{definition}
Given an initial $\Phi(\bm0,t)\in\mathcal{A}$ the vertex based \emph{moving frame} $\Phi:\mathbb{Z}^n\times\mathbb{R}\to\mathcal{A}$ of a parallelogram net $p$ is given by $\Phi_i=P^i\Phi$. If $\Phi$ is invertible the family of edge based nets $p_t$ given by
\begin{align}
p^i_t=\Phi_i^{-1}(P^i)'\Phi
\label{eq:assoedge}
\end{align}
is called the \emph{associated family} of $p$.
\end{definition}

Here, $(P^i)'=\lambda'+\mu'p^i$ denotes the derivative with respect to $t$. Note, that for any fixed $t$ the edge based nets $P$ defined by $P^i=\lambda+\mu p^i$ and $P'$ defined by $(P^i)'=\lambda'+\mu'p^i$ are parallelogram nets primarily equivalent to $p$.

\begin{proposition}
\label{thm:assofamily1}
Every net $p_t$ in the associated family of a parallelogram net $p$ is a parallelogram net. A primitive map $f_t$ for $p_t$ is given by $f_t=\Phi^{-1}\Phi'$.
\end{proposition}

Taking the logarithmic derivative to obtain a family of surfaces from its Lax representation is known as the Sym formula \cite{sym2005soliton}.
\begin{proof}
Consider the vertex based net $f_t=\Phi^{-1}\Phi'$. It is the primitive map of $p_t$, since
\begin{align*}
(f_t)_i-f_t&=\Phi^{-1}(P^i)^{-1}(P^i\Phi'+(P^i)'\Phi)-\Phi^{-1}\Phi'=\Phi_i^{-1}(P^i)'\Phi=p^i_t.
\end{align*}
The existence of this primitive map already implies the additive condition $p^i_t+(p^j_t)_i=p^j_t+(p^i_t)_j$. The multiplicative condition $(p^j_t)_ip_t^i=(p^i_t)_jp_t^j$ holds since
\begin{align*}
(p^j_t)_ip_t^i&=\Phi_{ij}^{-1}(P^j)_i'\Phi_i\Phi_i^{-1}(P^i)'\Phi=\Phi_{ij}^{-1}(P^j)_i'(P^i)'\Phi \quad\text{and}\\
(p^i_t)_jp_t^j&=\Phi_{ij}^{-1}(P^i)_j'\Phi_j\Phi_j^{-1}(P^j)'\Phi=\Phi_{ij}^{-1}(P^i)_j'(P^j)'\Phi.
\end{align*}
We see that these coincide using \eqref{eq:pareqmult} for the map $P'$ which is parallelogram net for any $t$. Hence, $p_t$ is a parallelogram net.
\end{proof}

Since we have the functional freedom of choosing $\lambda$ and $\mu$ it might appear as if we can get a lot of different nets from this construction. However, we can show that, up to primary equivalence, each net $p_{t_0}$ is already determined by the ratio $\frac{\lambda(t_0)}{\mu(t_0)}$:

\begin{proposition}
\label{thm:assofamily2}
For arbitrary functions $\lambda,\mu$ consider $p_{t_0}$ in the associated family of a parallelogram net $p$. If $\mu(t_0)=0$ the net $p_{t_0}$ is primarily equivalent to $p$. If $\mu(t_0)\neq0$ the net $p_{t_0}$ is primarily equivalent to $\tilde p_r$ where $r=\frac{\lambda(t_0)}{\mu(t_0)}$ and $\tilde p_t$ is the associated family defined by $\tilde\lambda(t)=t,\tilde\mu(t)=1$ of $p$.
\end{proposition}

Thus, for $\mathbb{K}=\mathbb{R}$ the associated family is a one-parameter family and different choices of $\lambda$ and $\mu$ correspond to different parametrizations of (subsets of) this family. For $\mathbb{K}=\mathbb{C}$ the associated family depends on the path of the ratio $\frac{\lambda}{\mu}$ in the complex plane and in Section \ref{sec:surfaces} this will be the real line or the unit circle.

\begin{proof}
We assume $\Phi(\bm0,t)=\tilde\Phi(\bm0,t)=1$ for for simplicity.

First, consider $t_0$ with $\mu(t_0)=0$. The nets $p$ and $p_{t_0}$ are primarily equivalent, since $P^i(t_0)=\lambda(t_0)$ implies $\Phi(x,t_0)=\lambda(t_0)^{x_1+\cdots+x_n}\in\mathbb{K}$ for $x=(x_1,\cdots,x_n)\in\mathbb{Z}^n$ and, hence, at $t=t_0$ we have
\begin{align*}
p^i_{t_0}=\Phi_i^{-1}(P^i)'\Phi=\frac{\lambda'}{\lambda}+\frac{\mu'}{\lambda}p^i.
\end{align*}

Now, consider $t_0$ with $\mu(t_0)\neq 0$. 
Note, that we have $\tilde{P}^i(t)=t+p^i$ and $(\tilde{P}^i)'(t)=1$. 
Then, we have $P^i(t_0)=\mu(t_0)(\frac{\lambda(t_0)}{\mu(t_0)}+p^i)=\mu(t_0)(r+p^i)=\mu(t_0)\tilde P^i(r)$ and, hence, $\Phi(x,t_0)=\mu(t_0)^{x_0+\cdots+x_n}\tilde\Phi(x,r)$. Now, $p_{t_0}$ and $\tilde p_{r}$ are primarily equivalent, since
\begin{align*}
p^i_{t_0}=\Phi_i^{-1}(P^i)'\Phi=\frac1\mu\tilde\Phi_i^{-1}(\lambda'+\mu'p^i)\tilde\Phi
=\frac1\mu\tilde\Phi_i^{-1}(\mu'\tilde P^i+\lambda'-r\mu')\tilde\Phi=\frac{\mu'}{\mu}+\frac{\lambda'-r\mu'}{\mu}\tilde p^i_{r}.
\end{align*}
Here, $\lambda,\mu,P,\Phi$ are evaluated at $t_0$ and $\tilde P,\tilde\Phi$ are evaluated at $r$.
\end{proof}

We conclude the discussion on general associative algebras by arguing that, generically, any linear Lax representation is equivalent to \eqref{eq:laxpair} under a gauge. For this consider a vertex based frame $\Psi$ transported along each edge by $\Psi_i=(\lambda A^i+\mu B^i)\Psi$. A gauge is a transformation $\Psi\to\Phi=G\Psi$ for a vertex based $G$ and in this case we define $G:=\Psi(\lambda=1,\mu=0)^{-1}$. Then, $G_i=G(A^i)^{-1}$ and the new frame $\Phi$ fulfils
\begin{align*}
\Phi_i=G_i\Psi_i=G(A^i)^{-1}(\lambda A^i+\mu B^i)G^{-1}\Phi=(\lambda+\mu p^i)\Phi,
\end{align*}
where $p^i:=G(A^i)^{-1}B^iG^{-1}$ defines a parallelogram net. The associated family \eqref{eq:assoedge} is not affected by this gauge since $G$ does not depend on $t$.

\begin{remark}
For example, the Lax representation for indefinite affine sphere in \cite{discreteAffineSpheres} is linear in $\lambda$ (with $\mu=1$). One can gauge it into \eqref{eq:laxpair} after a reparametrization $\lambda\to\lambda+1$ which is necessary for $A^i$ to be invertible.
\end{remark}

\subsection{Special algebras: \texorpdfstring{$\mathbb{C}^{2\times2}$}{Complex matrices} and the quaternions}
\label{sec:algebras}

To understand the role of parallelogram nets in discrete differential geometry it is useful to look at specific associative algebras $\mathcal{A}$. If $\mathcal{A}$ is commutative, e.g.,  $\mathcal{A}=\mathbb{K}$, the evolution is the identity and the quadrilaterals become actual planar parallelograms. Simple non-commutative cases are the complex matrix algebra $\mathbb{C}^{2\times2}$ and its subalgebra isomorphic to the quaternions $\mathbb{H}$. Most geometric structures investigated in this paper can be described within these algebras.

The invariants of the parallelogram evolution \eqref{eq:parevol} in $\mathbb{C}^{2\times2}$ are trace and determinant: If the net is evolvable they have the labelling property, i.e., for $i\neq j$ we have
\begin{align*}
	\tr p^i=\tr p^i_j,\qquad \det p^i=\det p^i_j.
\end{align*}

Conditions on the invariants can lead to interesting special cases:

\begin{definition}
	We call a parallelogram net $p$ \emph{equally folded} if $\det p^i=1$ for all edges. We call it \emph{zero-folded} if $\tr p^i=0$ for all edges.
\end{definition}

We will motivate this terminology when we turn to the subalgebra of quaternions. Note, that an evolvable zero-folded quad is contained in a two-dimensional subspace since its diagonals are linearly dependent:
\begin{align*}
p^i+p^j_i=(p^i(p^j-p^i)+(p^j-p^i)p^j)(p^j-p^i)^{-1}=\frac{\det p^i-\det p^j}{\det(p^j-p^i)}(p^i-p^j).
\end{align*}

\begin{remark}
Although this is not obvious, one can consider zero-folded to be a limit case of equally folded: To each quad of a parallelogram net $p$ we assign the number $r:=\frac{\det p^j-\det p^i}{\tr p^i-\tr p^j}$ which is the unique solution of $\det(r+p^i)=\det(r+p^j)$. Now, as a generalization of equally folded, assume that $r$ is a global constant. Then, $s:=\sqrt{\det (r+p^i)}$ also is constant and $p$ is primarily equivalent to the equally folded net defined by $\tilde p^i=\frac1s(r+p^i)$. The limit case of global $r=\infty$ appears if $\tr p^i$ is constant. Then, $p$ is primarily equivalent to the zero-folded net defined by $\tilde p^i=-\frac12\tr p^i+p^i$.
\end{remark}

Equally folded nets come in families:
\begin{theorem}
	\label{thm:zerofoldedrep}
	Using the functions $\lambda(t)=e^t,\mu(t)=e^{-t}$, a net in the associated family of a zero-folded net is equally folded. Every equally folded net can be represented this way.
\end{theorem}

According to theorem \ref{thm:assofamily2}, other choices of $\lambda$ and $\mu$ yield parallelogram nets that are primarily equivalent to equally folded nets.

\begin{proof}
	If $p$ is zero-folded, we can simply calculate
	\begin{align*}
		\det p^i_t=\det\Phi^{-1}(P^i)^{-1}(P^i)'\Phi=\frac{\det (e^t-e^{-t}p^i)}{\det (e^t+e^{-t}p^i)}=\frac{e^{2t}+e^{-2t}\det p^i}{e^{2t}+e^{-2t}\det p^i}=1.
	\end{align*}
	
	If $p$ is equally folded, we can find a zero-folded net $\hat{p}$ in its associated family: Consider $p_{t=1}$ in the associated family obtained from $\lambda(t)=t^2,\mu(t)=1$. Then, $\hat{p}:=p_{t=1}-1$ can be shown to be zero-folded and - with $\lambda=e^t,\mu=e^{-t}$ - one can calculate $\hat{p}_{t=0}=p$.
\end{proof}

We now study parallelogram nets in the quaternions. The quaternions $\mathbb{H}$ are the elements of the real division algebra generated by $1,\bm i,\bm{j},\bm k$ which fulfill $\bm i^2=\bm j^2=\bm k^2=\bm{ijk}=-1$. They can be represented as the real subalgebra of the complex matrices generated by
\begin{align*}
1\cong\begin{pmatrix}1 & 0 \\ 0 & 1\end{pmatrix},\quad
\bm i\cong\begin{pmatrix}0&\ci\\\ci&0\end{pmatrix},\quad
\bm j\cong\begin{pmatrix}0&1\\-1&0\end{pmatrix},\quad
\bm k\cong\begin{pmatrix}\ci&0\\0&-\ci\end{pmatrix}.
\end{align*}

For a quaternion $q\in\mathbb{H}$, we denote its real and imaginary parts by
\begin{align*}
r_q:=\Re(q)=\frac{1}{2}\tr q, \qquad \vec{q}:=\Im(q).
\end{align*}
We call an edge based net $p$ \emph{quaternionic} if every edge variable $p^i$ is a quaternion. Euclidean three-space can be identified with imaginary quaternions $\Im(\mathbb{H})\cong\mathbb{R}^3$. If $f:\mathbb{Z}^n\to\mathbb{H}$ is a vertex based net we will view $\vec f:\mathbb{Z}^n\to \Im(\mathbb{H})$ as a net in Euclidean space.

The invariants of a quaternionic parallelogram net are the real part $r_{p^i}$ and norm $\|p^i\|=\sqrt{\det p^i}$. Therefore, the length of each edge of a quad in Euclidean space $\|\vec p^i\|$ is also an invariant. Hence, these Euclidean quads are Chebyshev quads, which are quads where opposite edges have equal length. In particular, if $r_{p^i}=r_{p^j}$ the corresponding quad is an \emph{anti-parallelogram}, which is a planar Chebyshev quad with parallel diagonals. 

\begin{remark}
	The only non-evolvable solution is $p^i=p^j$ and $p^i_j=p^j_i$. Here, real part and length are not necessarily invariant and this case will not be considered in this paper.
\end{remark}

\begin{remark}
\label{rem:foldingparam}
Quaternionic equally folded nets are given by \cite[equations (3.16),(3.17)]{schief2007chebyshev} as special discrete Lund-Regge surfaces. For a geometric interpretation, the \emph{folding parameter} of a single skew parallelogram in $\R^3$ is given in~\cite{hoffmannSmokeRingFlow,hoffmann16} as
\begin{align*}
	\sigma:=\frac{\sin\delta^i_j}{\|\vec{p^i}\|}=\frac{\sin\delta^j_i}{\|\vec{p^j}\|},
\end{align*}
where $\delta^i_j$ is the dihedral angle at the edge $p^i$ of the tetrahedron spanned by $\vec f,\vec f_i,\vec f_j,\vec f_{ij}$. Equally folded nets can be calculated to have $|\sigma|=1$ on every quad, motivating our terminology.
\end{remark}

As the quaternions are a real algebra, the spectral parameters $\lambda$ and $\mu$ will, usually, be real numbers. Then, the Lax representation \eqref{eq:laxpair} and associated family are quaternionic. However, since the quaternions are a subset of $\mathbb{C}^{2\times2}$ the complex Lax representation also is well-defined. For example, in section \ref{sec:elastica} we will consider $P^j(\ci,1)=\ci+p^j$ (here, $\ci$ is the imaginary unit) for a quaternionic parallelogram net.

We can study how the geometry evolves in the (real) associated family. The Euclidean direction of an edge is given by $\frac{\vec p^i}{\|\vec p^i\|}$.

\begin{proposition}
\label{lem:vertexstars}
At each point the associated family of a quaternionic parallelogram net acts as a rotation on the set of Euclidean directions incident to the point.
\end{proposition}

In particular, all angles at a point stay the same.

\begin{proof}
From \eqref{eq:assoedge} we know
\begin{align*}
p^i_t=\Phi^{-1}(P^i)^{-1}(P^i)'\Phi,\qquad (p^i_t)_{\bar{i}}=\Phi^{-1}(P^i)_{\bar{i}}^{-1}(P^i)'_{\bar{i}}\Phi.
\end{align*}
Taking the logarithmic derivative of $P^i$ does not affect its direction, since $\Span( 1,p^i)\subseteq \mathbb{H}$ is closed under multiplication. All directions are then determined by the sandwich product with $\Phi$ which acts as a rotation in $\mathbb{R}^3$.
\end{proof}

For example, if the directions incident to a point are coplanar the net is said to have \emph{planar vertex stars}. Parallelogram nets with this property turn out to be K-nets (see section \ref{sec:knets}) and Proposition \ref{lem:vertexstars} implies that K-nets stay K-nets in the associated family (for example, see Fig. \ref{fig:amsler}).
	
\subsection{Polynomials and (re)factorization}
\label{sec:refac}

Equations \eqref{eq:parevol} and \eqref{eq:parevolbw} give us a `forward' and `backward' evolution for parallelogram nets. 
But how about a `sidewards' evolution: Can we find $p^i$ and $p^j_i$ from $p^j$ and $p^i_j$? One formulation of this question is if we can find different factorizations of the polynomial
\begin{align}
\label{eq:quadpoly}
\mathcal{P}=(1+\mu p^i_j)(1+\mu p^j)=(1+\mu p^j_i)(1+\mu p^i).
\end{align}

\begin{figure}[h!]
  \centering
  \includegraphics[width=0.19\textwidth]{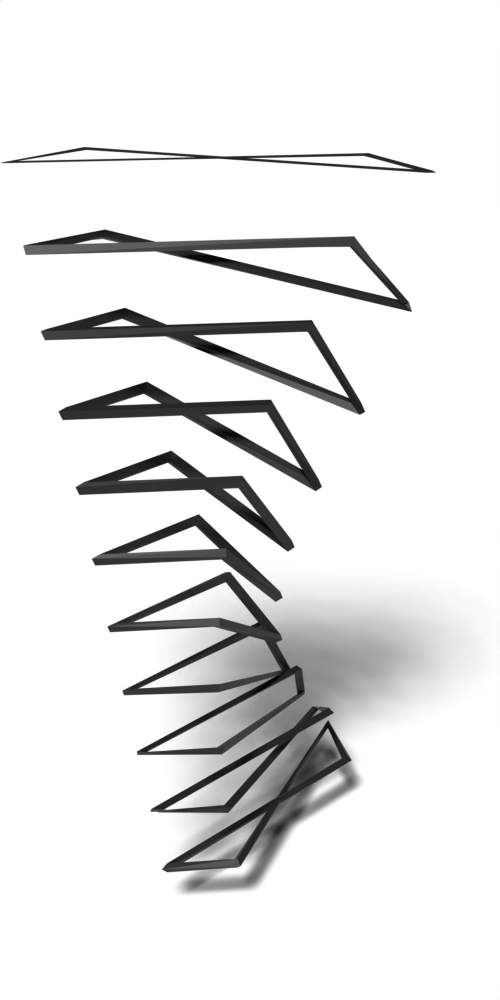}
  \hspace{12mm}
  \includegraphics[width=0.23\textwidth]{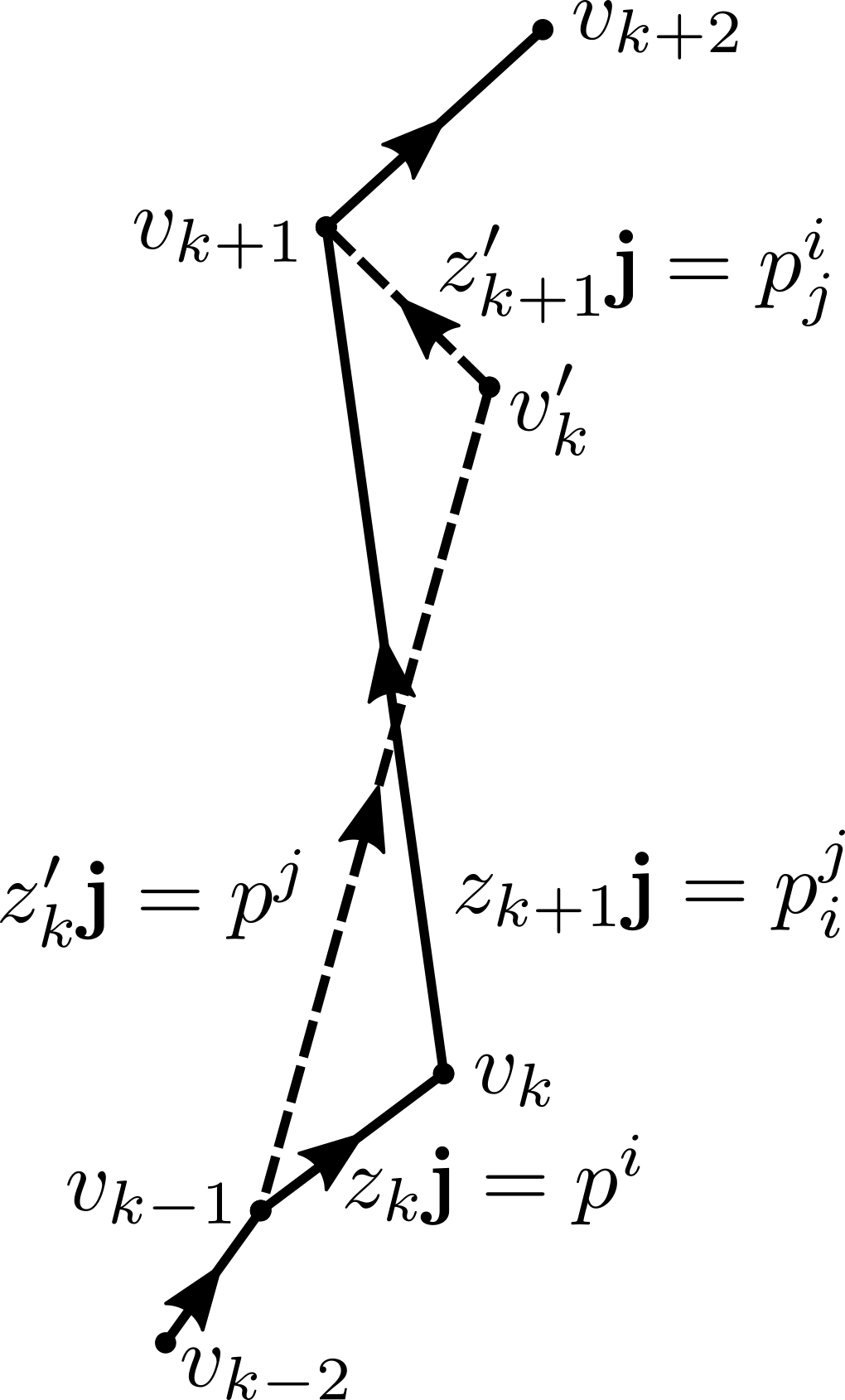}
  \hspace{10mm}
  \includegraphics[width=0.26\textwidth]{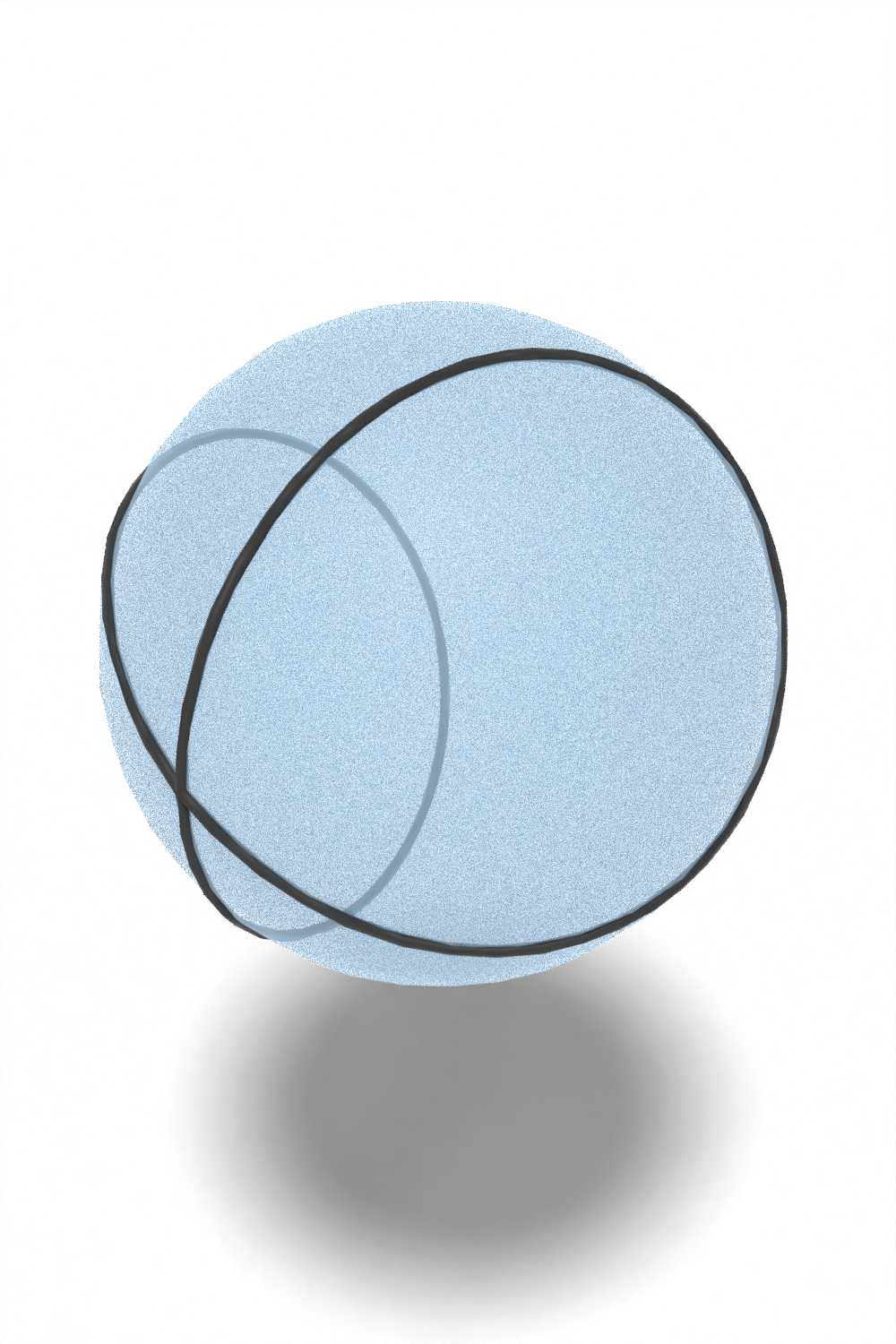}
  \caption{Three geometries from one solution of \eqref{eq:quadpoly}. Left: An anti-parallelogram with associated family. Middle: Recutting of a polygon. Right: Trace of the motion of a point.
  }
  \label{fig:threegeometries}
\end{figure}

Our focus is on how parallelogram nets, their evolvability, and factorizations like \eqref{eq:quadpoly}, relate to curves and surfaces. Before proceeding, we remark that these properties are also relevant in discrete dynamical systems and linkage kinematics. In Figure~\ref{fig:threegeometries} we showcase three different interpretations of \eqref{eq:quadpoly}.

\begin{remark}
Polygon recutting at a vertex $v_k$ of a sequence $(v_i\in\mathbb{C})_{i\in\mathbb{Z}}$ can be expressed in terms of its edges $z_k=v_k-v_{k-1}$ by (see~\cite[equation (19)]{izosimov2023recutting})
\begin{align*}
z_k'+z_{k+1}'&=z_k+z_{k+1},\\
z_k'\bar z_{k+1}'&=z_k\bar z_{k+1}.
\end{align*}
If we identify $\mathbb{C}$ with $\Span( 1,\bm i)\subseteq\mathbb{H}$ this is equivalent to
\begin{align*}
(1+\mu z_{k+1}'\bm j)(1+\mu z_k'\bm j)=(1+\mu z_{k+1}\bm j)(1+\mu z_k \bm j).
\end{align*}
After a conjugation, which reverses the order, and with $t=-\mu$ this equation agrees with~\cite[equation (9)]{izosimov2023recutting}. Its also agrees with \eqref{eq:quadpoly} if one defines $p^i=z_k\bm j,...$ as in Figure~\ref{fig:threegeometries} Middle.

With more involved combinatorics, one can express polygon recutting in terms of parallelogram nets even globally. For a finite sequence $v_0,\cdots,v_n$ consider an $n$-dimensional cube in a parallelogram net such that
\begin{align}
\label{eq:recutztop}
p^k(e_1+\cdots+e_{k-1})=z_k\bm j.
\end{align}
Recutting of the sequence then corresponds to taking different paths through the n-dimensional cube in the parallelogram net.

Now, for an infinite sequence $(v_i\in\mathbb{C})_{i\in\mathbb{Z}}$ closed up to translation we have $z_{k+n}=z_{k}$. Then, consider an $n$-dimension parallelogram net with \eqref{eq:recutztop} and the periodicity $p^k(x+e_1+\cdots e_n)=p^k(x)$. Any curve obtained by recutting $(v_k)$ can be found on a path in the parallelogram net.
\end{remark}

\begin{remark}
Polynomials like \eqref{eq:quadpoly} describe motions of points in euclidean space by $\mu\mapsto \mathcal{P}(\mu) q \mathcal{P}^{-1}(\mu)$ for $q\in \Im(\mathbb{H})$ and can, thus, be called \emph{motion polynomials}. For $\mathcal{P}(\mu)\in\mathbb{H}$ the motion is restricted to a sphere as visualized for one point in Figure~\ref{fig:threegeometries} Right. More generally, in~\cite{hegedus2013factorization} motion polynomials have coefficients in the dual quaternions which also allow for general Euclidean displacements. Linkages with rotational joints can be described using factorization since each linear factor describes one such joint. Two different factorizations for the same motion as in \eqref{eq:quadpoly} imply the existence of two linkages. This allows for the construction of a closed linkage that realizes this motion. In fact, any path through an $n$-dimensional cube in a parallelogram net corresponds to a linkage that realizes the same motion.
\end{remark}

We will investigate this factorization problem in a more general setting. To each sequence of consecutive edges $e(1),...,e(n)\in\vec{ \mathcal{E}}(\mathbb{Z}^n)$ we can assign the product of Lax matrices $\mathcal{P}=(1+\mu p(e(n)))\cdots(1+\mu p(e(1)))$, which is a polynomial of degree $n$ in the spectral parameter $\mu$. A natural question is if such polynomials can be factorized back into Lax matrices of a parallelogram net and if, hence, polynomial Lax representations can be found to be reductions of the parallelogram system.

Factorization of such polynomials has originally been studied for quaternions in~\cite{gordon1965zeros,niven1941equations} (see also~\cite{izosimov2023recutting}). Similar factorization results exist for dual quaternions~\cite{hegedus2013factorization}, split quaternions~\cite{scharler2020quadratic,scharler2021algorithm} and, recently, conformal geometric algebra~\cite{li2023geometric}. Some universal results can be found in~\cite{li2019factorization}. We will state a result for complex matrices and show how to obtain the well-known result for the subalgebra of quaternions. In particular, we give an explicit formula for the factor.

Given some polynomials over the complex numbers $a,b,c,d\in\mathbb{C}[\mu]$ we call $\mathcal{P}:=\begin{pmatrix}a & b \\ c & d\end{pmatrix}$ a \textit{matrix polynomial}. It can also be written as $\mathcal{P}=\sum_{i=0}^n \mu^iC_i$ with $C_i\in\mathbb{C}^{2\times2}$. Note, that the indeterminate $\mu$ is a scalar and, as such, commutes with the coefficients. Still, if one prescribes whether to multiply from the right or from the left one can evaluate the polynomial at matrices. For instance, a matrix $u\in\mathbb{C}^{2\times2}$ is called \emph{right root} (or \emph{right zero}) of $\mathcal{P}$ if $\sum_{i=0}^n C_i u^i=0$. 
If $\mathcal{P},\mathcal{R}$ are two matrix polynomials then $\mathcal{R}$ is called \emph{right factor} of $\mathcal{P}$ if there exists a third polynomial $\mathcal{Q}$ such that $\mathcal{P}=\mathcal{Q}\mathcal{R}$. A known result for polynomials over rings (see, e.g., \cite{li2019factorization}) is that a linear factor $\mu-u$ is a right factor if and only if $u$ is a right root. For invertible $u$ this is equivalent to $1+\mu\hat u$ being a right factor for $\hat u=-u^{-1}$. Left roots and factors can be defined analogously but we will stick to factorization from the right. 
In the following, we will use the adjugate matrix:
\begin{align*}
\adj\begin{pmatrix} a & b \\ c & d \end{pmatrix}=\begin{pmatrix} d & -b \\ -c & a \end{pmatrix}
\end{align*}
For a preliminary computation consider a matrix polynomial $\mathcal{P}$ with a right root $u\in\mathbb{C}^{2\times2}$ such that $\mathcal{P}=\mathcal{Q}(\mu-u)$. Then,
\begin{align*}
\mathcal{P}\adj(\mu-u)=\mathcal{Q}(\mu-u)\adj(\mu-u)=\mathcal{Q}\det (\mu-u).
\end{align*}
The eigenvalues $\mu_1,\mu_2\in\mathbb{C}$ of $u$ are the roots of $\det (\mu-u)$ and we find
\begin{align}
\label{eq:factorev}
(\mu_i-u)\adj\mathcal{P}(\mu_i)=0,\qquad i=1,2.
\end{align}
Thus, non-vanishing columns of $\adj\mathcal{P}(\mu_i)$ are eigenvectors of $u$. Each matrix yields at most one eigenvector (up to scaling) since $\det\mathcal{P}(\mu_i)=\det\mathcal{Q}(\mu_i)\det(\mu_i-u)=0$. 
If taking a column of each matrix determines a basis of eigenvectors $Y$ we can write $u$ in its diagonalized form
\begin{align}
\label{eq:diagzero}
u=Y\begin{pmatrix}\mu_1 & 0 \\ 0 & \mu_2\end{pmatrix}Y^{-1}
\end{align}
We formulate the existence of such a basis for a general matrix polynomial $\mathcal{P}$:

\begin{definition}
Let $\mathcal{P}$ be a matrix polynomial. We call two roots $\mu_1,\mu_2\in\mathbb{C}$ of $\det \mathcal{P}$ \emph{independent} if $\ker \mathcal{P}(\mu_1)\cap\ker \mathcal{P}(\mu_2)=\{(0,0)\}$. 
\end{definition}

For two independent roots we find linear independent columns $y_1$ and $y_2$ of $\adj\mathcal{P}(\mu_1)$ and $\adj\mathcal{P}(\mu_2)$ respectively and the matrix $Y=\begin{pmatrix}y_1 & y_2\end{pmatrix}$ is invertible.

\begin{remark}
Note, that if $\mu_1=\mu_2$ the pair is not independent and a corresponding right root of $\mathcal{P}$ can not be reconstructed using \eqref{eq:diagzero}. This includes all non-diagonalizable right roots. On the other hand, if such a right root is diagonalizable then it is a multiple of the identity and we have $\mathcal{P}(\mu_1)=\mathcal{P}(\mu_2)=0$. Thus, the scalar $\mu_1=\mu_2$ is not only a root of $\det\mathcal{P}$ but also of $\mathcal{P}$ itself. In particular, if $\mu_i$ is a scalar root of $\mathcal{P}$ it is not independent to any other root $\mu_j$ of $\det\mathcal{P}$.
\end{remark}

\begin{theorem}
\label{thm:factthm}
Let $\mathcal{P}$ be a matrix polynomial with two independent roots $\mu_1,\mu_2\in\mathbb{C}$ of $\det \mathcal{P}$. Then, $\mathcal{P}$ has a unique right factor of the form $\mu-u$ such that $\mu_1,\mu_2$ are the roots of $\det (\mu-u)$. For a corresponding basis $Y$ it is given by \eqref{eq:diagzero}.
\end{theorem}

\begin{proof}
Uniqueness follows from the above argumentation since any right root $u$ with independent eigenvalues $\mu_1,\mu_2$ is determined by \eqref{eq:diagzero}. Note, that since the columns of each $\adj\mathcal{P}(\mu_i)$ are linearly dependent the specific choice of the basis $Y$ does not affect $u$. 
For existence we need to show that \eqref{eq:diagzero} determines a right factor. We have the eigenvalue equations $uy_i=\mu_i y_i$ which immediately imply \eqref{eq:factorev}. Therefore, the polynomial $\mathcal{P}\adj(\mu-u)$ vanishes at $\mu_1,\mu_2$ and, hence, has $\det(\mu-u)$ as scalar factor. This defines a polynomial $\mathcal{Q}$ with
\begin{align*}
\mathcal{P}\adj(\mu-u)=\mathcal{Q}\det (\mu-u)=\mathcal{Q}(\mu-u)\adj(\mu-u)
\end{align*}
and we conclude $\mathcal{P}=\mathcal{Q}(\mu-u)$.
\end{proof}

This theorem allows us to find right factors of the polynomial as long as the independence condition is met. We can factorize the polynomial into linear factors if the condition is met in every step. Under the (possibly strong) condition that every pair of roots is independent in every step of the factorization we have $\frac{(2n)!}{2^n}$ possible factorizations belonging to the possibilities of choosing the pairs of roots. In most applications we only consider refactorizations which respect the pairs of roots, in other words, which are obtained by reordering a prescribed set of pairs of roots. 
Then, there are $n!$ such factorizations which belong to different paths of an $n$-dimensional parallelogram cube and on each quad we have the parallelogram equations \eqref{eq:pareqadd} and \eqref{eq:pareqmult}. 
Note, that factorization into linear factors of the form $\mu-u$ and $1+\mu u$ is the same since once can reparametrize the polynomial $\mathcal{P}(\mu)\to \mu^n\mathcal{P}(-\frac{1}{\mu})$. 
Hence, we conclude that the polynomial \eqref{eq:quadpoly} of a single quad has, generically, a total of six factorizations, belonging to the different possibilities to choose pairs out of the four roots of its determinant. If refactorization respects the pairs of roots we have the labelling property of trace and determinant and can explicitly compute the corresponding unique refactorization as

\begin{align}
\label{eq:evolbackwards}
p^i=(p^i_j-\adj p^j)p^i_j(p^i_j-\adj p^j)^{-1},\nonumber\\
p^j_i=(p^i_j-\adj p^j)p^j(p^i_j-\adj p^j)^{-1}.
\end{align}

We now turn to the factorization of quaternionic polynomials, i.e., polynomials in a real (or complex) parameter where the coefficients are quaternions. We are interested in factorizations where the factors are also quaternionic. 
Such a polynomial can be written as $\mathcal{P}=a\bm 1+b\bm i+c\bm j+d\bm k$ where $a,b,c,d\in\mathbb{R}[\mu]$ are real polynomials. 
Let us first consider the special case of a scalar root $\mu_0\in\mathbb{C}$ of $\mathcal{P}$ which means that $\mu-\mu_0$ is a common factor of $a,b,c$ and $d$. If $\mu_0\in\mathbb{R}$ is real we can split off the scalar factor and the quotient is still a quaternionic polynomial. If $\mu_0\in\mathbb{C}$ is complex this is not the case: The quotient does not remain quaternionic after division of $\mathcal{P}$ by the scalar factor.  However, since $a,b,c,d$ are real polynomials the complex conjugate $\mu-\bar\mu_0$ is also a scalar factor of $\mathcal{P}$ and, therefore, we find $(\mu-\mu_0 )(\mu-\bar\mu_0)$ as a real quadratic factor of $\mathcal{P}$. There exist infinitely many quaternionic factorizations of this real polynomial since
\begin{align*}
(\mu-\mu_0 )(\mu-\bar\mu_0)=\mu^2-2\mu\Re_\mathbb{C} \mu_0+|\mu_0|^2=(\mu-u )(\mu-\bar u)
\end{align*}
holds for any quaternion $u$ with $\Re_\mathbb{H}  u=\Re_\mathbb{C} \mu_0$ and $\det u=|\mu_0|^2$. Since the space of such possible roots $u$ is a sphere such roots are called \emph{spherical roots} of $\mathcal{P}$. A quaternionic polynomial with no real polynomial factors has no spherical roots.

\begin{theorem}
\label{thm:facquat}
Let $\mathcal{P}$ be a quaternionic polynomial with no real polynomial factors and let $\mu_0\in\mathbb{C}$ be a root of $\det \mathcal{P}$. Then, there exists a unique right factor $\mu-u$ with $u\in\mathbb{H}$ such that $\mu_0$ is a root of $\det (\mu-u)$.
\end{theorem}

Thus, any quaternionic polynomial can be factorized into real factors, factors belonging to spherical roots and factors given as in the theorem. This result is well-known. We present a proof as special case of Theorem \ref{thm:factthm} which allows for a direct computation of right roots using \eqref{eq:diagzero}.

\begin{proof}
To $\mathcal{P}=a\bm 1+b\bm i+c\bm j+d\bm k$ the determinant polynomial is the real polynomial $\det\mathcal{P}=a^2+b^2+c^2+d^2$ and its roots come in pairs related by complex conjugation. Thus, we set $\mu_1=\mu_0$ and $\mu_2=\bar\mu_0$. We will show that they are independent by finding a basis $Y$ given by a column of $\adj\mathcal{P}(\mu_0)$ and $\adj\mathcal{P}(\bar\mu_0)$ respectively: For this, note that,
\begin{align*}
\adj\mathcal{P}=\begin{pmatrix}a-id & -c-ib \\ c-ib & a+id\end{pmatrix}.
\end{align*}
Consider the matrices
\begin{align*}
\begin{pmatrix}a_0-id_0 & -\bar c_0-i\bar b_0 \\ c_0-ib_0 & \bar a_0+i\bar d_0\end{pmatrix},\qquad \begin{pmatrix}-c_0-ib_0 & -\bar a_0+i\bar d_0 \\ a_0+id_0 & -\bar c_0+i\bar b_0\end{pmatrix}
\end{align*}
where $a_0:=a(\mu_0),b_0:=b(\mu_0),\hdots$. The first matrix consists of the first column of $\adj\mathcal{P}(\mu_0)$ and the second column of $\adj\mathcal{P}(\bar\mu_0)$ while the second matrix consists of the second column of $\adj\mathcal{P}(\mu_0)$ and minus the first column of $\adj\mathcal{P}(\bar\mu_0)$. Both matrices are quaternions. Also, one of them does not vanish: If both vanish we have $a_0=b_0=c_0=d_0=0$ and, therefore, $\mu_0,\bar\mu_0$ are scalar roots of $\mathcal{P}$. Then, $(\mu-\mu_0 )(\mu-\bar\mu_0)$ is a real factor of $\mathcal{P}$ which contradicts our assumptions. We define $Y$ to be this non-vanishing matrix. As a quaternion, it is invertible. Therefore, $\mu_0$ and $\bar\mu_0$ are independent and Theorem \ref{thm:factthm} gives a unique right factor $\mu-u$. Since $u$ is given by \eqref{eq:diagzero} and $Y$ is a quaternion $u$ is also a quaternion.
\end{proof}

All possible quaternionic factorizations belong to an $n$-dimensional parallelogram cube:

\begin{corollary}
A quaternionic polynomial $\mathcal{P}$ of degree $n$ with no real polynomial factors has $n!$ possible factorizations of the form
\begin{align*}
\mathcal{P}=C^n(\mu-u^n)\cdots(\mu-u^1)
\end{align*}
Each factorization belongs to a path of an $n$-dimensional cube and on all quads of this cube the roots $u^i$ fulfill the parallelogram equations \eqref{eq:pareqadd} and \eqref{eq:pareqmult}. 
If all complex roots of $\det\mathcal{P}$ are distinct then all factorizations are different and the cube is evolvable.
\end{corollary}

We conclude this section with an example showing how one can also consider non-quaternionic factorizations of quaternionic polynomials.

\begin{example}
\label{ex:twoFacs}
As a consequence of the two factorization theorems, we look at the quaternionic polynomial
\begin{align*}
\mathcal{P}:=1+\mu\begin{pmatrix}0 & 2\\-2 & 0\end{pmatrix}+\mu^2\begin{pmatrix}-2 & 0\\0 & -2\end{pmatrix}.
\end{align*}
We know that we can find a quaternionic factorization, e.g.,
\begin{align*}
\mathcal{P}=(1+\mu\begin{pmatrix}1 & 1\\-1 & 1\end{pmatrix})(1+\mu\begin{pmatrix}-1 & 1\\-1 & -1\end{pmatrix}).
\end{align*}
By a different pairing of the roots we can also find a factorization which is not quaternionic, e.g.,
\begin{align*}
\mathcal{P}=(1+\mu\begin{pmatrix}0 & 1+\ci\\-1-\ci & 0\end{pmatrix})(1+\mu\begin{pmatrix}0 & 1-\ci\\-1+\ci & 0\end{pmatrix}).
\end{align*}
Unlike the quaternionic factorization, the second factorization is zero-folded. We will present an application of this in section \ref{sec:cknets}.
\end{example}

\section{Invariant curves and elastic rods}
\label{sec:elastica}

In this section we study curves invariant under a sequence of Bäcklund transformations, in other words, curves for which such a sequence only induces a Euclidean motion. Curves invariant under two Bäcklund transformations have been shown to be elastic rods~\cite{hoffmannSmokeRingFlow}. Here, we also show the converse using factorization: To each elastic rod we find Bäcklund transformations that leave the curve invariant.

\subsection{Bäcklund transformations of curves}

We consider Bäcklund transformations of discrete curves in Euclidean space. Such a curve is a map $\gamma:\mathbb{Z}\to \Im(\mathbb{H})\cong\mathbb{R}^3$. It is the primitive map of the edge based map $u=\gamma_1-\gamma\in \Im(\mathbb{H})$, which trivially forms a $0$-dimensional parallelogram net. We will make no distinction between the curve and the edge based parallelogram net.

We assume parametrization by arclength, i.e., $|u|=1$. Some of our results can be generalized to arbitrary parametrization, however, the explanations becomes a lot more technical and will be presented in~\cite{mythesis}. We call the curve regular if $\gamma_{\bar1}\neq\gamma_1$.
\begin{figure}[h!]
  \center
  \includegraphics[width=0.26\textwidth]{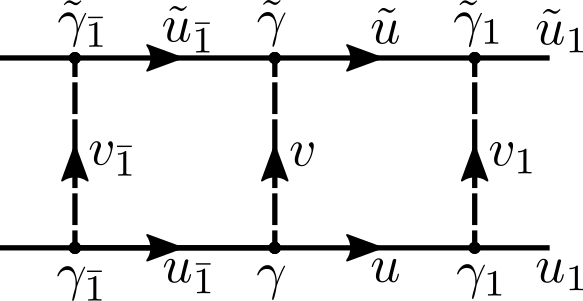}
  \caption{Notation for a Bäcklund transformation of a curve.}
  \label{fig:baecklundnot1}
\end{figure}

We will write the assumptions on real part and length of the edges $u$ as 
\begin{align*}
u\in\mathbb{S}^2=\{p\in\mathbb{H} \, | \, r_p=0, |p|=1\}=\{p\in\mathbb{H} \, | p^2=-1\}.
\end{align*}

For a Bäcklund transformation we will also assume $v\notin\mathbb{S}^2$ where $v$ is an edge variable in transformation direction. This guarantees evolvability, since $v\neq u$ everywhere. Non-evolvability can have interesting consequences as also explained in~\cite{mythesis}.

Bäcklund transformations of curves in terms of the parallelogram equations have been used before. In particular, in~\cite{pinkallSmokeRingFlow} and~\cite{hoffmannSmokeRingFlow} they are used to discretize the Hashimoto flow of a discrete curve. They noticed that the parallelogram evolution of an edge $u\in\mathbb{S}^2$ with $v\in\mathbb{H}\setminus\mathbb{S}^2$ can be written as a fractional linear map:
\begin{align*}
T_v:\mathbb{S}^2\to\mathbb{S}^2,\quad u\mapsto(v-u)u(v-u)^{-1}=(vu+1)(-u+v)^{-1}.
\end{align*}
If $\tilde u$ is a Bäcklund transform of a curve $u$ with $v=\tilde{\gamma}-\gamma\in\mathbb{H}\setminus\mathbb{S}^2$ we have
\begin{align*}
\tilde{u}=T_v (u),\qquad \tilde{u}_{\bar{1}}=T_v (u_{\bar{1}}).
\end{align*}
The second equation is just the 'backwards' evolution \eqref{eq:evolbackwards} with $u=p^i,v=p^j$. Such a map $T_v$ is linear when written in homogeneous coordinates in $\mathbb{HP}^1$, i.e., $\mathbb{HP}^1\ni(p,q)\cong pq^{-1}\in\mathbb{H}$. Then, the map is given by
\begin{align*}
u\cong\begin{pmatrix}u \\ 1\end{pmatrix}\mapsto \begin{pmatrix}v & 1 \\ -1 & v\end{pmatrix}\begin{pmatrix}u \\ 1\end{pmatrix}\cong T_v(u).
\end{align*}
We will make no distinction between the map $T_v$ and its matrix representation in $\mathbb{HP}^1$. The matrix $T_v$ belongs to the set of matrices of the form
\begin{align}
\label{eq:ABmatrix}
\begin{pmatrix}A & B \\ -B & A\end{pmatrix}\in\mathbb{H}^{2\times2}.
\end{align}
As described with more detail in~\cite{pinkallSmokeRingFlow}, this set becomes an algebra isomorphic to $gl(2,\mathbb{C})$ when equipped with the scalar multiplication
\begin{align*}
\lambda M:=(\Re_{\mathbb{C}}(\lambda)I+\Im_{\mathbb{C}}(\lambda)J)M,\qquad \lambda\in\mathbb{C},\, I=\begin{pmatrix}1 & 0\\0 & 1\end{pmatrix},\, J=\begin{pmatrix}0 & 1 \\ -1 & 0\end{pmatrix}.
\end{align*}
The matrix $J$ corresponds to the imaginary unit $\ci\in\mathbb{C}$ and, in particular, commutes with all elements of the set. Complex trace and determinant are given by
\begin{align*}
\frac12\tr \begin{pmatrix}A & B \\ -B & A\end{pmatrix}=r_A+r_B \ci,\quad \det \begin{pmatrix}A & B \\ -B & A\end{pmatrix}=|B|^2-|A|^2-2\langle A,B\rangle_4 \ci.
\end{align*}
where $\langle A,B\rangle_4=\Re(A\cdot adj(B))$ denotes the standard scalar product of the quaternions. In particular, we have $\det T_v=\|v\|^2-1+2r_v \ci\neq0$ for $v\notin\mathbb{S}^2$. The trace free component of $M=\begin{pmatrix}A & B \\ -B & A\end{pmatrix}$ is $\begin{pmatrix}\vec A & \vec B \\ -\vec B & \vec A\end{pmatrix}$ and we will, again, denote it by $\vec{M}$.

Note, that $\det M=0$ is equivalent to either $B^{-1}A\in\mathbb{S}^2$ or $A=B=0$. If $\det M\neq0$, the matrix $M$ defines a map $\mathbb{S}^2\to\mathbb{S}^2$:
\begin{align*}
u\cong\begin{pmatrix}u \\ 1\end{pmatrix}\mapsto M\begin{pmatrix}u \\ 1\end{pmatrix}\cong(A-Bu)u(A-Bu)^{-1}\in\mathbb{S}^2.
\end{align*}
It is well defined, since $A= Bu$ for a $u\in\mathbb{S}^2$ implies $\det M=0$. If $\vec M=0$ this map acts as the identity. If $B=0$ it acts as a conjugation: $u\mapsto AuA^{-1}$.

Matrices with $B=0$ are isomorphic to the quaternions. Our map $T_v$ is just a representation of $\ci+v$. Hence, it fulfills the parallelogram equations \eqref{eq:pareqadd} and \eqref{eq:pareqmult}. Similarly, the Lax matrix in~\cite{pinkallSmokeRingFlow} is a representation of $\lambda+\ci v$ and is of the form \eqref{eq:laxpair}.

\subsection{Invariant curves}

Now, we are interested in curves that stay invariant under a sequence of Bäcklund transformations: Consider a curve $\gamma$ and a sequence of $n$ Bäcklund transformations, such that the $n$-th transform $\tilde{\gamma}$ is related to $\gamma$ by a Euclidean motion. The Bäcklund transformations form a parallelogram net $p=(u,v)$ and we write $u(k)=u(k,0)$ and $\tilde{u}(k)=u(k,n)$ (see Figure~\ref{fig:baecklundnot2}).

\begin{remark}
If the Euclidean motion is the identity then $v(k,0),\cdots,v(k,n)$ forms a closed polygon for each $k$ and neighbouring polygons are related by a Bäcklund transformation. This structure has been studied in~\cite{pinkallSmokeRingFlow,hoffmannSmokeRingFlow} and one can interpret this chapter as generalization of their results.
\end{remark}

\begin{figure}[h!]
  \center
  \includegraphics[width=0.25\textwidth]{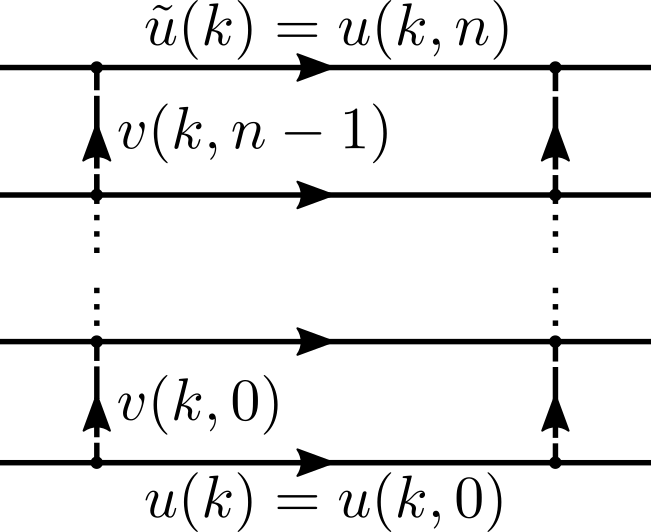}
  \caption{Notation for multiple Bäcklund transformations.}
  \label{fig:baecklundnot2}
\end{figure}
 
The Euclidean motion acts as a rotation on $u$ which we express by a quaternion $E$ fulfilling $u=E\tilde{u}E^{-1}$. Locally, this rotation composed with the Bäcklund transformations can be written as a linear map in $\mathbb{HP}^1$:
\begin{align}
\label{eq:Mbaecklunds}
M(k):=\begin{pmatrix}A(k) & B(k) \\ -B(k) & A(k)\end{pmatrix}:=\begin{pmatrix}E & 0 \\ 0 & E\end{pmatrix}T_{v(k,n-1)}\cdots T_{v(k,0)}.
\end{align}

Note, that since $\det T_{v(k,l)}\neq 0$ implies $\det M\neq0$ we always have a well-defined map. Still, there is one degenerate case we would like to exclude: $M$ can become the identity, i.e., $\vec M=0$. This can happen, e.g., if $E,v(0,0),...,v(0,n-1)\in\mathbb{R}$. Less trivial examples of this can be obtained by going back and forth or by using 3D-consistency. If $M$ is not the identity, invariance under Bäcklund transformations is an interesting condition on the curve:

\begin{definition}
We call a regular arc-length parametrized curve $\gamma$ \emph{$n$-invariant} if there exists a sequence of $n$ Bäcklund transformations of $\gamma$, such that the $n$-th transform $\tilde{\gamma}$ is related to $\gamma$ by a Euclidean motion for which $M$ defined by \eqref{eq:Mbaecklunds} is not the identity.
\end{definition}

Note, that the existence of identity Bäcklund transformations ($v\in\mathbb{R}$) implies that an $n$-invariant curve is also an $m$-invariant curve for $m\geq n$.

\begin{remark}
One can show that $n$-invariance is preserved in the associated family. This will be discussed in detail in~\cite{mythesis}.
\end{remark}

To study invariants of such curves, consider the polynomial
\begin{align}
\mathcal{P}(k,\lambda):=E(\lambda+ v(k,n-1))\cdots(\lambda+ v(k,0)).
\end{align}
From the parallelogram equations and $E(\lambda+\tilde u)=(\lambda+ u)E$ we can obtain
\begin{align}
\label{eq:ninvpolycomp}
\mathcal{P}_1(\lambda)(\lambda+ u)=(\lambda+ u)\mathcal{P}(\lambda).
\end{align}
This implies that the polynomials $\det\mathcal{P}(\lambda)$ and $\tr\mathcal{P}(\lambda)$ are constant along the curve and its coefficients yield real invariants. Here, $\det\mathcal{P}$ is the product of $\det E$ and $\det(\lambda+v(k,l))=(\lambda+r_{v(k,l)})^2+\det\vec v(k,l)$ and is, thus, determined by $\det E$ and the invariants of the parallelogram evolution. On the other hand, the coefficients of $\tr\mathcal{P}(\lambda)$ are $r_E$ and further $n$ non-trivial invariants which we can understand for $n=2$ in section \ref{sec:elasticrods}. In particular, since we know that $T_v$ is just a representation of $\ci+v$ we conclude that $M$ coincides with $\mathcal{P}(\ci)$ and, hence, $\det M$ and $\frac12\tr M=r_A+\ci r_B$ are invariants.
\begin{remark}
For $E=1$ these invariants coincide with the invariants in~\cite{pinkallSmokeRingFlow}. In fact, any dynamic that acts on such a polynomial via conjugation leaves the determinant and trace invariant. Both dynamics~\cite{izosimov2023recutting} and~\cite{tabachnikov2012discrete} are of this type as specified in~\cite[Remark 4.13]{izosimov2023recutting}.
\end{remark}

For an $n$-invariant curve, the map $M$ has $u$ and $-u_{\bar{1}}$ as fixed points, since it is the composition of the Bäcklund transformations mapping $u\mapsto\tilde u$ and the Euclidean motion mapping $\tilde{u}\mapsto u$ (similarly for $-u_{\bar{1}}$). This observation is the key to construct an algorithm which gives $n$-invariant curves:

\begin{theorem}[Algorithm]
\label{thm:algo1}
Let $E\in\mathbb{H}\setminus\{0\}$, $v(0,0),...,v(0,n-1)\in\mathbb{H}\setminus\mathbb{S}^2$, such that $\det\vec{M}(0)\neq0$. We can iterate the following steps, starting from $i=0$:
\begin{enumerate}
\item If $i=0$, choose $u(0)$ as one of the two fixed points of $M(0)$.
If $i>0$, choose $u(i)$ as the unique fixed point of $M(i)$ for which $u(i)\neq -u(i-1)$.
\item Use the parallelogram evolution to construct $v(i+1,0),...,v(i+1,n-1)$.
\end{enumerate}
Then, the curve $\gamma$ with edges $u$ is an n-invariant curve. Any $n$-invariant curve can be obtained this way.
\end{theorem}

To actually define a curve on $\mathbb{Z}$ one can extend the curve to $i<0$ by using the other fixed point. We will disregard this for brevity.

\begin{figure}[h!]
  \centering
  \includegraphics[width=0.23\textwidth]{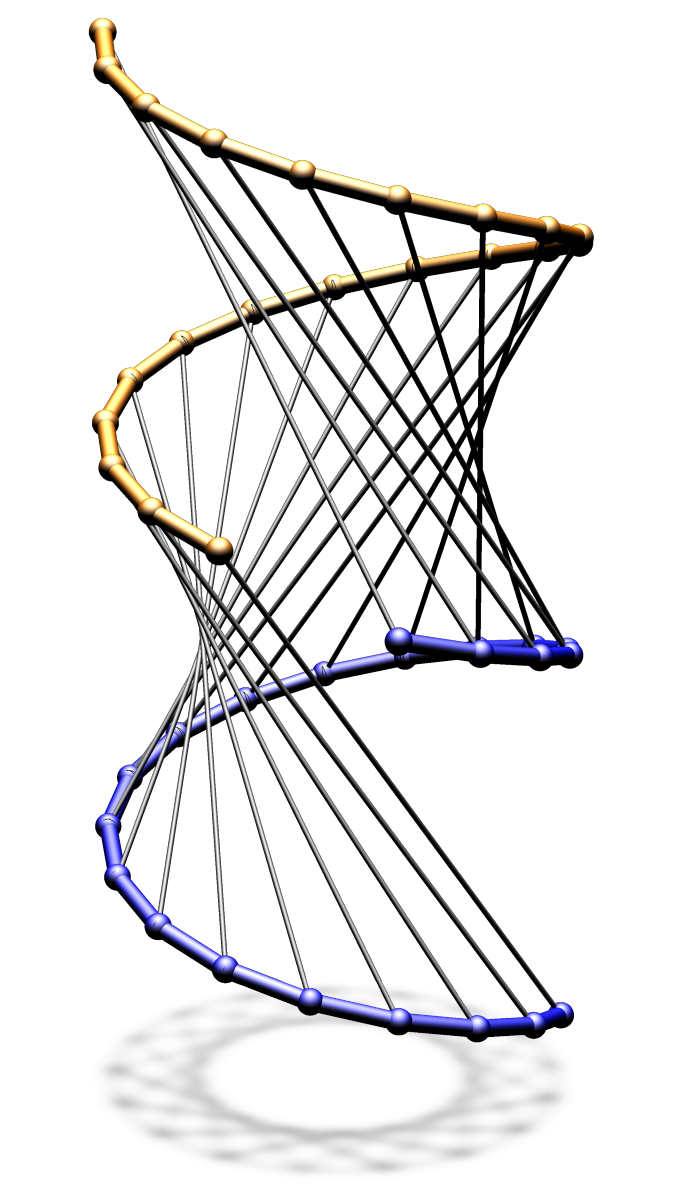}
  \includegraphics[width=0.34\textwidth]{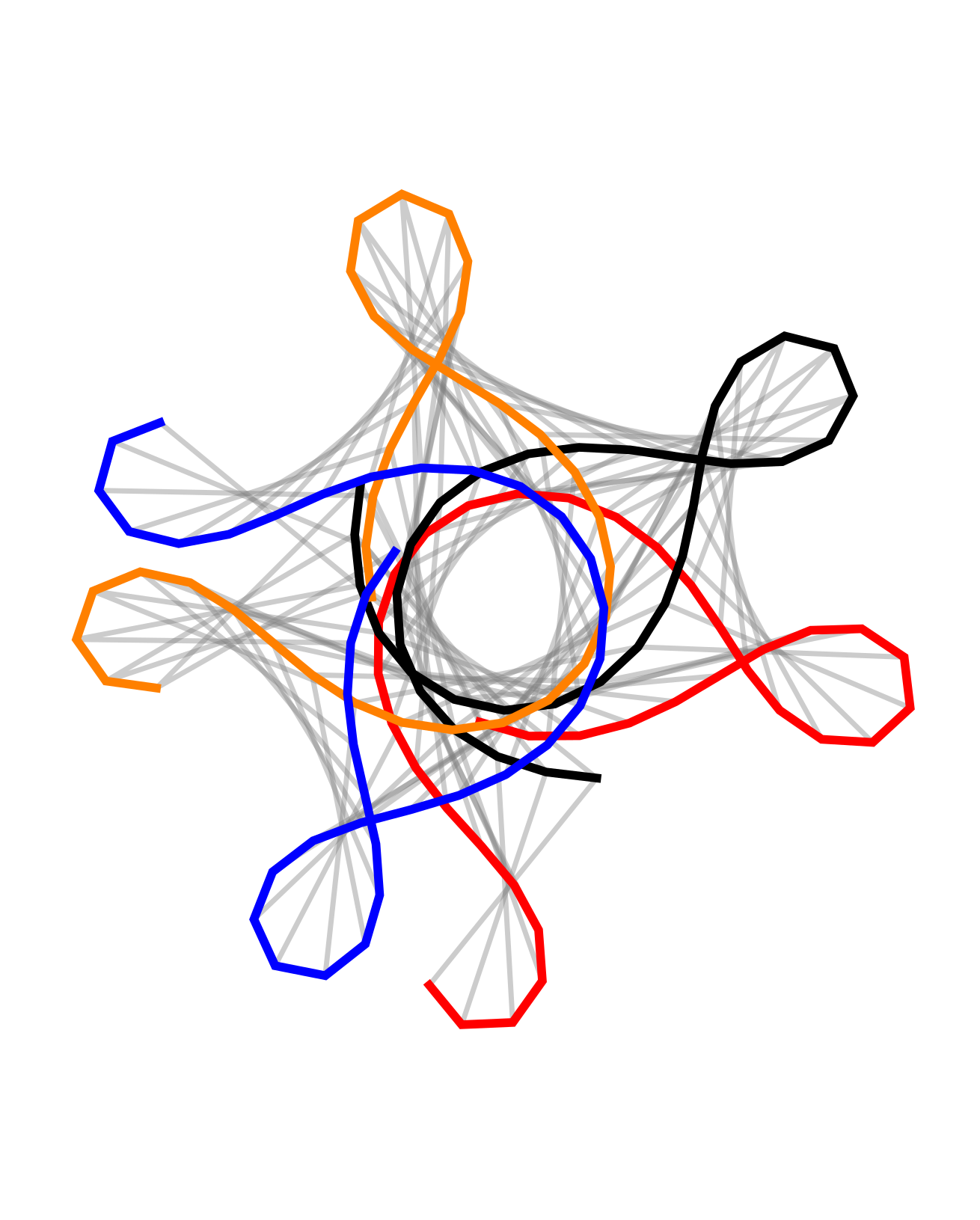}
  \includegraphics[width=0.36\textwidth]{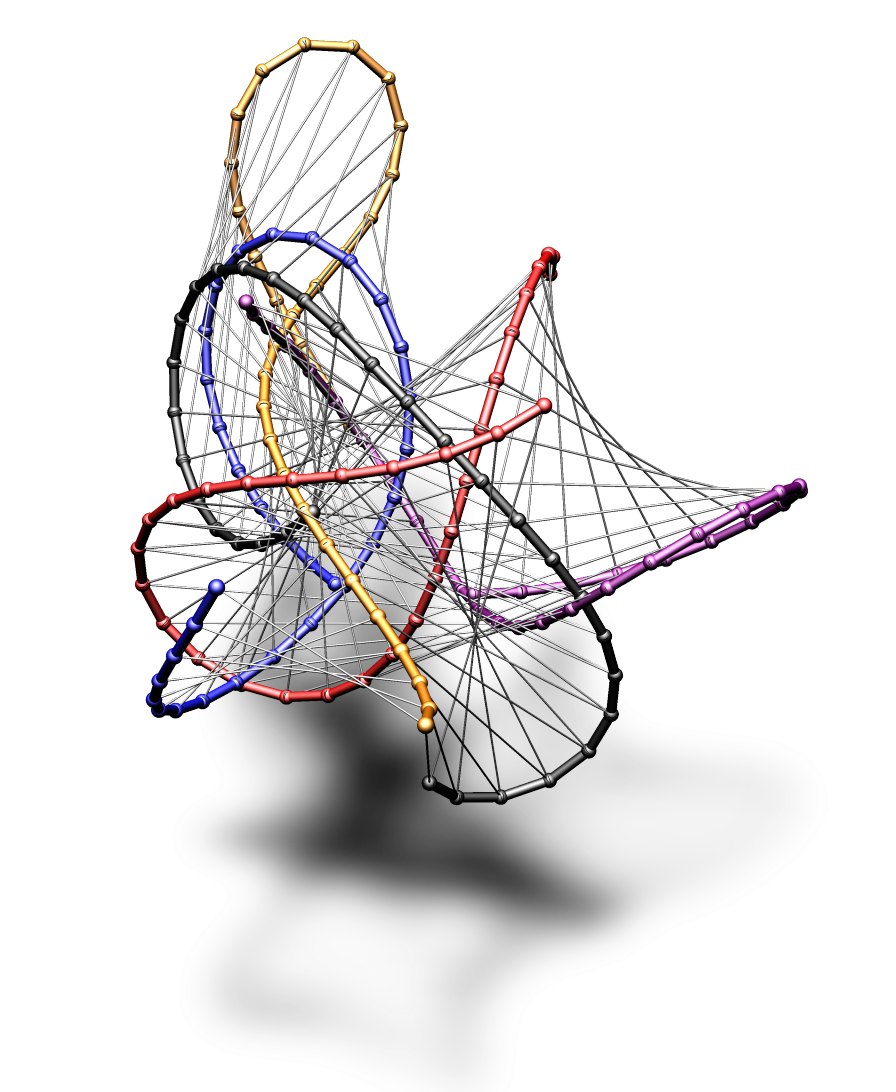}
  \caption{$n$-invariant curves with corresponding Bäcklund transformations for $n=1,3,4$. The 1-invariant curve is a helix. The 3-invariant curve is planar (see Remark \ref{rem:planarReduction}).}
  \label{fig:ninvariant}
\end{figure}

While we use our algorithm to construct a curve from initial Bäcklund data one can also apply it to construct closed Bäcklund transformations for a given curve as in~\cite{pinkallSmokeRingFlow,hoffmannSmokeRingFlow}.

For a proof we need to understand the fixed points of $M$.

\begin{lemma}
Consider a non-identity map $\mathbb{S}^2\to\mathbb{S}^2$ defined by $M=\begin{pmatrix}A & B \\ -B & A\end{pmatrix}\in\mathbb{H}^{2\times2}$ with $\det M\neq 0$. Its two fixed points are given by
\begin{align}
u=(\beta+\vec B)^{-1}(\alpha+\vec A),
\label{eq:evform}
\end{align}
where $\alpha,\beta\in\mathbb{R}$ fulfills $\alpha+\ci\beta=\pm\sqrt{\det\vec{M}}$.
\end{lemma}

Note, that $\det\vec M\neq 0$ excludes the identity and the case where both fixed points coincide.

\begin{proof}
We need to find every $u\in\mathbb{S}^2$ such that
\begin{align*}
(A-Bu)u=u(A-Bu).
\end{align*}
Two quaternions commute if and only if their imaginary parts are parallel. This means that there exists $\alpha,\beta\in\mathbb{R}$ such that
\begin{align*}
A-Bu=(r_A-\alpha)+(\beta-r_B) u,
\end{align*}
which can be solved into \eqref{eq:evform}.
The equations $r_u=0$ and $\|u\|=1$ can be written as
\begin{align*}
0&=r_u\|\beta+\vec B\|^2=\Re((\beta-\vec B)(\alpha+\vec A))=\beta\alpha+\langle\vec B,\vec A\rangle\\
0&=\|\beta+\vec B\|^2-\|\alpha+\vec A\|^2=\beta^2-\alpha^2+\|\vec B\|^2-\|\vec A\|^2
\end{align*}
These are equivalent to the complex equation
\begin{align*}
(\alpha+\ci\beta)^2=\|\vec B\|^2-\|\vec A\|^2-2\ci\langle\vec B,\vec A\rangle=\det\vec{M}.
\end{align*}
Hence $u$ is a fixed point if and only if it fulfills \eqref{eq:evform} with $\alpha,\beta$ of this form.
\end{proof}

Now we can prove Theorem \ref{thm:algo1}.

\begin{proof}[of theorem]
Clearly, all $n$-invariant curves can be obtained from the algorithm. Note, that regularity of the curve implies $\det\vec M\neq 0$, since distinct fixed points of $M$ exist.

Also, the algorithm always yields an $n$-invariant curve. Note,  that $\det\vec M\neq0$ gets preserved, since $\det M$ and $\tr M$ are invariants. Then, we have exactly two fixed points in every step and by construction one of them coincides with $-u(i-1,0)$.
\end{proof}

We can even show that $\alpha$ and $\beta$ given by $u=(\beta+\vec B)^{-1}(\alpha+\vec A)$ are constant along the curve. Equation \eqref{eq:ninvpolycomp} with $\lambda=\ci$ becomes $M_1T_u=T_uM$. The off-diagonal entries of this gives $A_1+B_1u=uB+A$. This already implies that
\begin{align*}
A_1+B_1u&=u(r_B+\vec B)+\vec A+\alpha-\alpha+r_A=u(r_B+\vec B)+(\beta+\vec B)u-\alpha+r_A\\
&=u\vec B+\vec Bu+(\beta+r_b)u-\alpha+r_a=(\beta+r_b)u+(\alpha+r_a),
\end{align*}
where in the last step we use $u\vec B+\vec B u=2Re(\hat Bu)=2\alpha$. Then,
\begin{align*}
-u&=(-\beta+\vec{B}_1)^{-1}(-\alpha+\vec{A}_1)
\end{align*}
and, hence, the pair $-\alpha,-\beta$ gives us the 'backwards' fixed point $-u$. The 'forwards' fixed point $u_1$ is the other fixed point and must be obtained with $\alpha_1=\alpha,\beta_1=\beta$. We define
\begin{align*}
\hat A:=\alpha+\vec A,\qquad \hat B:=\beta+\vec B
\end{align*}
which simplifies our formulas to
\begin{align}
\label{eq:ABsimple}
u=\hat B^{-1}\hat A,\quad u_{\bar{1}}=\hat A\hat B^{-1}\quad \Rightarrow \quad u=\hat B^{-1}u_{\bar{1}}\hat B=\hat A^{-1}u_{\bar{1}}\hat A.
\end{align}

\begin{remark}
\label{rem:planarReduction}
The algorithm allows for a reduction to the plane: If all $v(0,k)\in
\Span(\bm i,\bm j)$ and $E\in\Span(\bm i,\bm j)$ for even $n$ or $E\in \Span( 1,\bm k)$ for odd $n$ then we have $\det \vec M\in\mathbb{R}$. In this case if $\det\vec M<0$ one can show $u\in\Span(\bm i,\bm j)$ meaning that the curve and its transformations lie in a common plane. This reduction also is a special case of cross-ratio dynamics \cite{arnold2022cross}.
\end{remark}

\subsection{Elastic rods}
\label{sec:elasticrods}

In this section, we give more details about the case $n=2$, such as a simplification of the above algorithm. Two Bäcklund transformations acting on a curve are known to discretize the Hashimoto flow \cite{pinkallSmokeRingFlow,hoffmannSmokeRingFlow}. We show that 2-invariant curves are exactly elastic rods as given in~\cite{lagrangeTop}. One can apply similar techniques to other $n$. For $n=1$ the results are not as interesting, since we always obtain a discrete helix. For $n\geq3$ the techniques become more complicated and will be discussed elsewhere.

For $n=2$, $A$ and $B$ are
\begin{align*}
A&=Ev_2v-E,\\
B&=E(v+v_2).
\end{align*}
Note, that $A,B$ and $E$ appear in the polynomial
\begin{align*}
\mathcal{P}(\lambda):=\lambda^2E+\lambda B+(A+E)=E(\lambda+ v_2)(\lambda+ v).
\end{align*}
The parallelogram equations \eqref{eq:pareqadd},\eqref{eq:pareqmult} yield an evolution for $A$ and $B$ along the curve:
\begin{align}
\label{eq:ABevol1}
A_1&=E(v_{12}v_1-1)=E(\tilde{u}v_2vu^{-1}-1)=E(E^{-1}uE(v_2v-1+1)u^{-1}-1)\nonumber\\
&=uAu^{-1}+uEu^{-1}-E,\\
\label{eq:ABevol2}
B_1&=E(v_1+v_{12})=E(-u+v+v_2+\tilde{u})=E(-u+v+v_2+E^{-1}uE)\nonumber\\
&=B+uE-Eu.
\end{align}

We can again observe that $r_A,r_B,r_E$ are constants and don't affect this evolution. Hence, $\hat A$ and $\hat B$ have the same evolution.

Note, that we have $(\Re(\hat Bu),\Re(\hat B))=(\alpha,\beta)\neq(0,0)$, since that would imply $\det\vec M=0$.

We can now formulate a simplified algorithm.

\begin{theorem}[Simplified algorithm for 2-invariant curves]
\label{thm:algo2}
Let $E,\hat B(0)\in\mathbb{H}\setminus\{0\}$, $u(0)\in\mathbb{S}^2$, such that at $i=0$, we have $(\Re(\hat Bu),\Re(\hat B))\neq(0,0)$. We can iterate the following steps, starting from $i=0$:
\begin{enumerate}
\item Calculate $\hat B(i+1):=\hat B(i)+u(i)E-Eu(i)$.
\item Calculate $u(i+1):=\hat B^{-1}(i)u\hat B(i)$.
\end{enumerate}
Then, the curve $\gamma$ with edge variables $u$ is a 2-invariant curve. Any 2-invariant curve can be obtained this way.
\end{theorem}

Remarkably, as we will see in the proof, recovering the Bäcklund transformations turns out to be an application of the factorization theorem.

\begin{proof}
If we have a 2-invariant curve, we know it fulfills the evolutions for $\hat B$ and $u$. Therefore, it can be constructed by this algorithm.

For the converse, if we have a curve constructed by this algorithm, we can define $\hat{A}:=\hat{B}u$. $\hat{A}$ evolves as in \eqref{eq:ABevol1} by
\begin{align*}
\hat A_1=\hat B_1u_1=u\hat B_1=u\hat B+u^2E-uEu=u\hat Au^{-1}-E+uEu^{-1}.
\end{align*}
Unfortunately, for $\hat M:=\begin{pmatrix}\hat A &\hat B\\ -\hat B & \hat A\end{pmatrix}$ we have $\det\hat M=0$, and, hence, this can not be a map as in \eqref{eq:Mbaecklunds}. But we can alter the real parts, i.e., we can define $A:=a+\hat{A},B:=b+\hat B$ with $a,b\in\mathbb{R}$, such that $\det M\neq0$ for $M:=\begin{pmatrix}A & B \\-B & A\end{pmatrix}$. The new $A$ and $B$ still fulfill \eqref{eq:ABevol1} and \eqref{eq:ABevol2}. Now, we can apply the factorization theorem to the polynomial
\begin{align*}
\mathcal{P}(\mu):=E+\mu B(0)+\mu^2(A(0)+E)
\end{align*}
to obtain a quaternionic factorization from theorem \ref{thm:facquat}:
\begin{align*}
\mathcal{P}(\mu)=E(1+\mu v(0,1))(1+\mu v(0,0))
\end{align*}
The factors $v$ are not in $\mathbb{S}^2$, since this would have implied $\det M(0)=0$. Also, since $0=\det\hat M(0)=(\alpha+\ci\beta)^2+\det\vec M(0)$, we have $\det\vec M(0)\neq 0$.

We can apply algorithm \ref{thm:algo1} to $E,v(0,0),v(0,1),u(0)$ to obtain a 2-invariant curve. Our edge variables $u$ and $A,B$ must agree with the corresponding values of the new curve, since their evolution agrees and they agree at $i=0$. Hence, our curve $u$ is 2-invariant.
\end{proof}

The freedom in choosing $r_A$ and $r_B$ means that there are many Bäcklund transformations that leave the curve invariant. There is the third freedom to also choose $r_E$ which corresponds to choosing the angle of rotation of the Euclidean motion between $\gamma$ and $\tilde\gamma$. One can use these freedoms to get special Bäcklund transformations for a given 2-invariant curve. This will be presented in~\cite{mythesis}.

\begin{figure}[h!]
  \centering
  \includegraphics[width=0.3\textwidth]{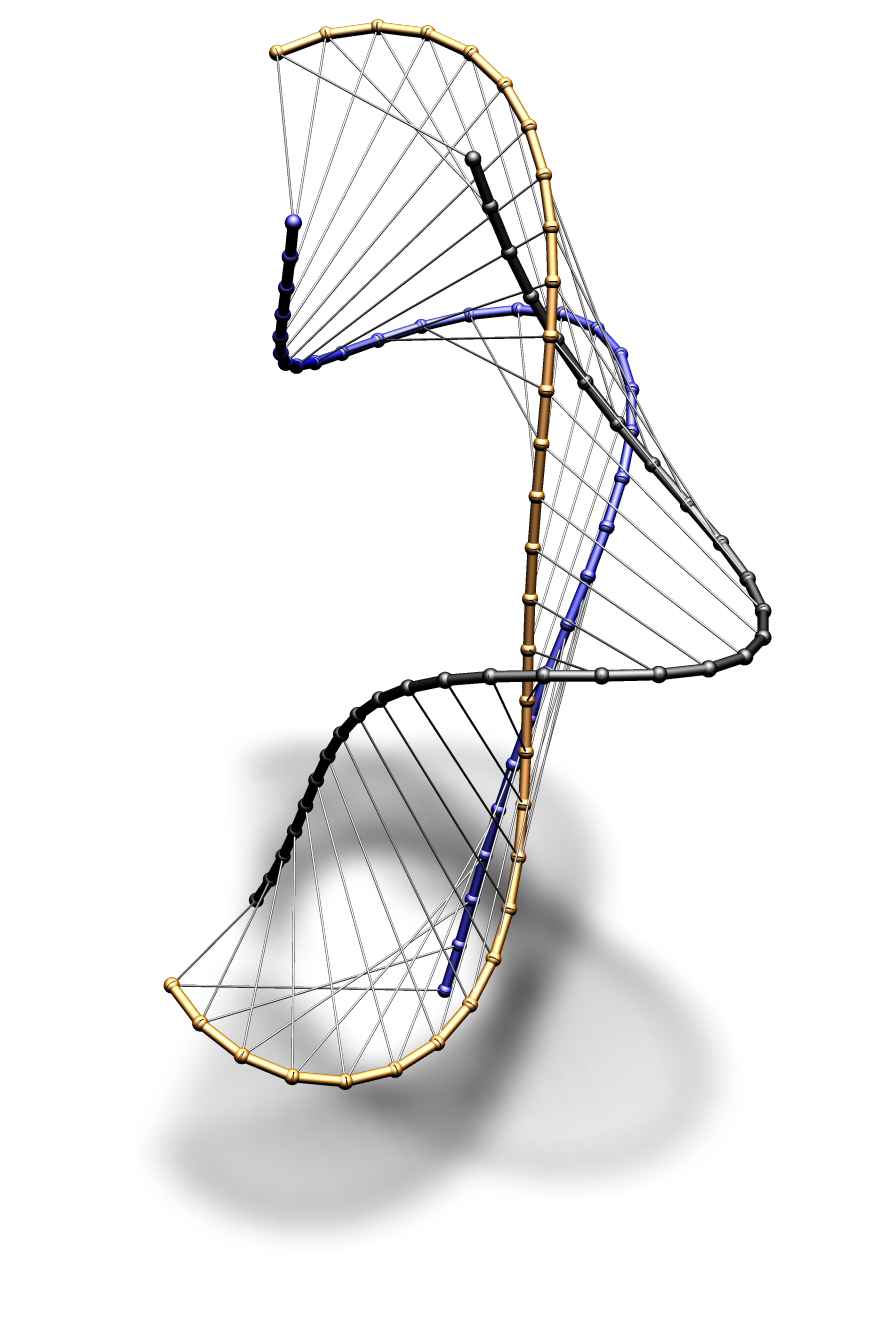}
  \includegraphics[width=0.31\textwidth]{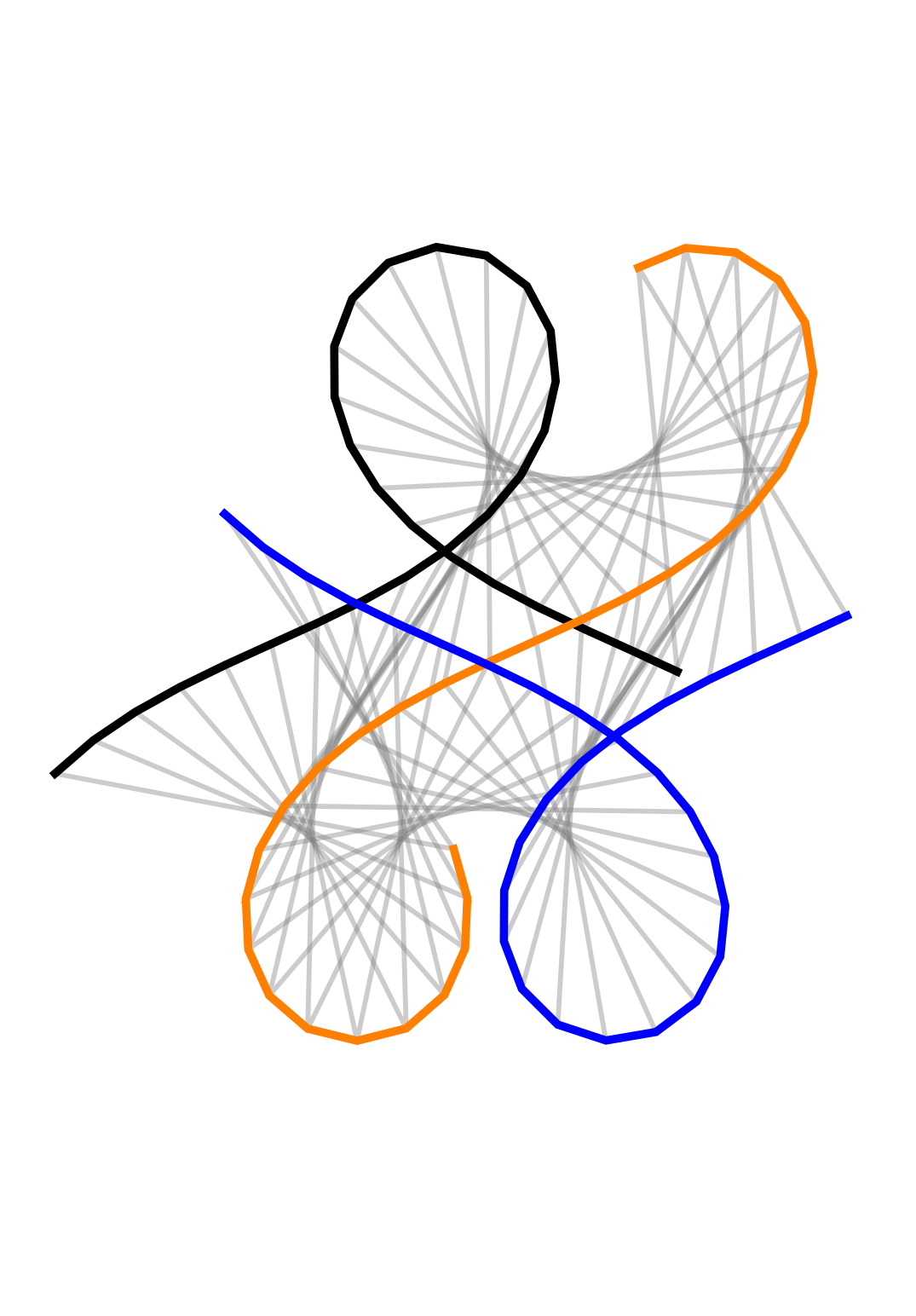}
  \includegraphics[width=0.33\textwidth]{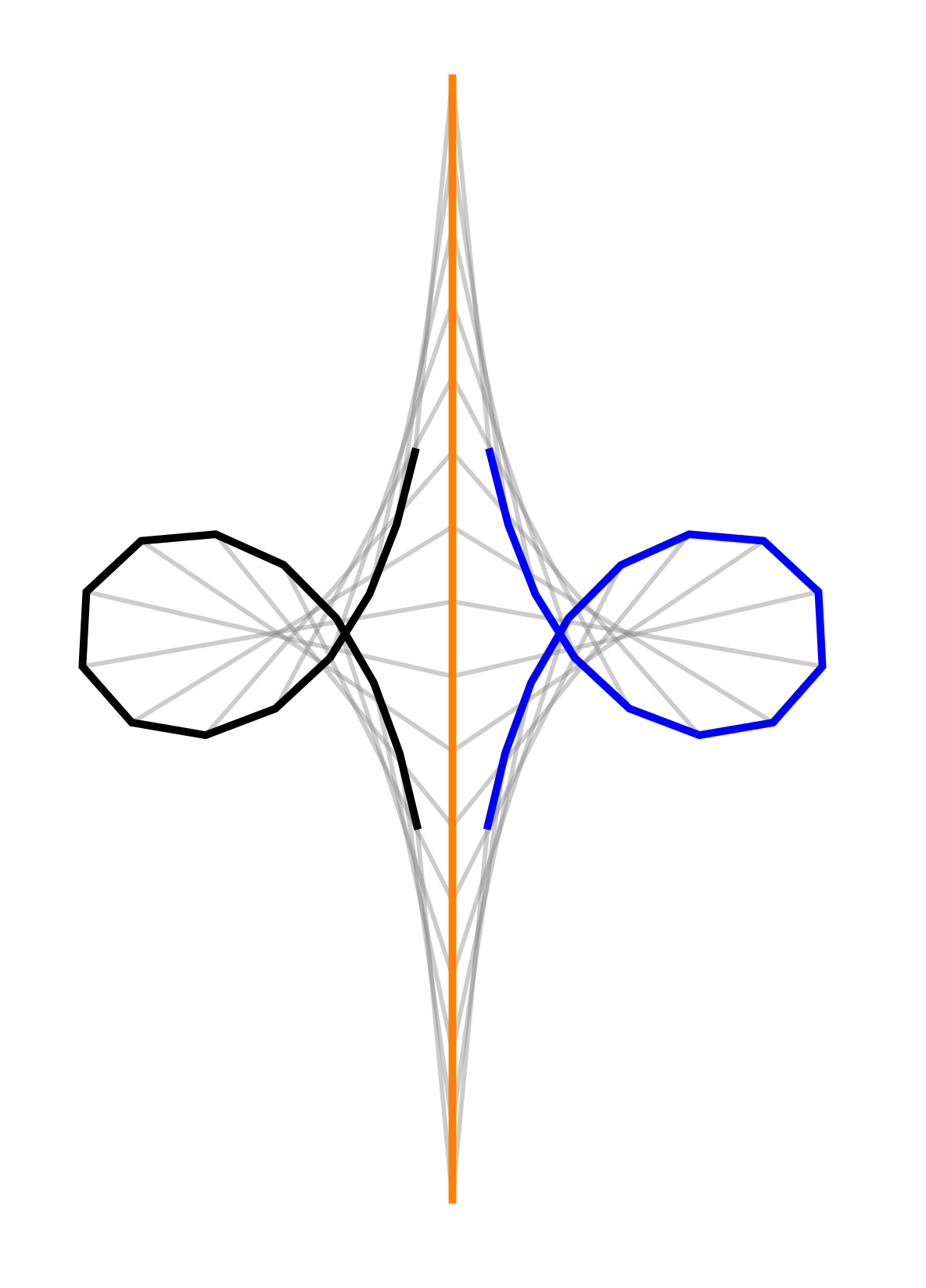}
  \caption{Elastic rods and their corresponding Bäcklund transformations. If all curves are in a common plane the first and last curve are related by reflection and translation. The euler loop is 2-invariant, since it can be transformed into the straight line which can then be transformed to a reflected euler loop.}
  \label{fig:elastica}
\end{figure}

We will now show that 2-invariant curves are precisely elastic rods as described in~\cite{lagrangeTop}. Elastica are usually described as curves minimizing bending energy under certain constraints. However, one can equivalently characterize them as curves that stay invariant under certain flows. In~\cite{lagrangeTop} it is shown that discrete elastic rods can be characterized as discrete curves that stay invariant under a combination of tangent flow $\dot{\gamma}=F^T$ and Hashimoto flow $\dot{\gamma}=F^H$ where
\begin{align*}
F^T=\frac{u_{\bar{1}}+u}{1+\langle u_{\bar{1}},u\rangle},\qquad F^H=2\frac{u_{\bar{1}}\times u}{1+\langle u_{\bar{1}},u \rangle}.
\end{align*}
A curve stays invariant under a flow if $\dot\gamma$ is the vector field of a Euclidean motion which in this case means that there exist constant $(\alpha,\beta)\in\mathbb{R}^2\setminus\{(0,0)\}$, $e,x\in\mathbb{R}^3$, such that
\begin{align}
\label{eq:elasticrod}
\dot\gamma=2\alpha F^T+\beta F^H=e\times\gamma+x.
\end{align}
 
\begin{theorem}
\label{thm:elasticrod}
A regular arc-length parametrized discrete curve is a 2-invariant curve if and only if stays invariant under a combination of tangent and Hashimoto flow, i.e., if it is an elastic rod as in~\cite{lagrangeTop}.
\end{theorem}

Since Bäcklund transformations can be seen as discrete flows this just means that curves invariant under the smooth flow are also invariant under the discrete flow and vice versa.

In~\cite{hoffmannSmokeRingFlow} it is already shown that curves that stay invariant under two Bäcklund transformations are elastic rods. Here, we also prove the converse and, in particular, know from algorithm \ref{thm:algo2} how one can reconstruct the Bäcklund transformations for a given elastic rod.

\begin{proof}
First, we express the two flows in terms of quaternions:
\begin{align*}
F^T&=-2(u_{\bar{1}}+u)^{-1},\\
\frac12F^H-1&=2\frac{u_{\bar{1}}\times u-(1+\langle u_{\bar{1}},u \rangle)}{|u_{\bar{1}}+u|^2}=2\frac{u_{\bar{1}}u+u^2}{|u_{\bar{1}}+u|^2}=-2(u_{\bar{1}}+u)^{-1}u.
\end{align*}

Now, assume $\gamma$ is 2-invariant. Then, there exists $\hat A,\hat B$ fulfilling \eqref{eq:ABevol1}, \eqref{eq:ABevol2} and \eqref{eq:ABsimple}. Now, we can calculate $\hat B$ from $u$:

\begin{align*}
(u_{\bar{1}}+u)\hat B&=(\hat Bu\hat B^{-1}+u)\hat B=\hat Bu+u\hat B=2\alpha+2\beta u,
\end{align*}
where in the last step, we use $u\vec B+\vec B u=2Re(\hat Bu)=2\alpha$. Then,
\begin{align*}
\hat B&=2(u_{\bar{1}}+u)^{-1}(\alpha+\beta u)=\beta-\alpha F^T-\frac{\beta}{2}F^H.
\end{align*}

The evolution for $\vec{B}$ is
\begin{align*}
\vec{B}_1-\vec{B}=uE-Eu=-2\vec{E}\times u=-2\vec{E}\times(\gamma_1-\gamma),
\end{align*}
which implies that there exist $x\in \Im(\mathbb{H})$, such that for $e=4\vec{E}$
\begin{align*}
-\frac12(2\alpha F^T+\beta F^H)=\vec{B}=-\frac12(e\times\gamma+x).
\end{align*}
This implies \eqref{eq:elasticrod}.

For the converse, assume there exist $\alpha,\beta\in\mathbb{R}$ and $e,x\in\mathbb{R}^3$ fulfilling \eqref{eq:elasticrod}. Define $\hat B:=\beta-\alpha F^T-\frac{\beta}{2} F^H$ and $\hat A:=\hat Bu$. By assumption, we have $\hat B_1-\hat B=uE-Eu$ for $E\in\mathbb{H}$ with $\vec{E}=\frac14e$. Also,
\begin{align*}
\hat B=2(u_{\bar{1}}+u)^{-1}(\alpha+\beta u).
\end{align*}
implies $u_{\bar{1}}\hat B=\hat Bu$ and, hence, the curve can be obtained from algorithm \ref{thm:algo2} and is, thus, a 2-invariant curve.
\end{proof}

If $\alpha=0$ the curve is known to be an elastic curve. By \eqref{eq:ABsimple}, $u_{\bar1},u,A$ are all orthogonal to $\vec B$. $A$ and $\vec B$ become tangent and binormal vector:
\begin{align*}
A \parallel F^T,\qquad \vec B \parallel F^H.
\end{align*}
If, additionally, $\vec B$ is also orthogonal to $\vec E$, the curve is a planar elastic curve.

We conclude this section with a few remarks on how the results in this section can be generalized. First, algorithms 1 and 2 can be extended to work for non arc-length parametrized curves. We believe such $2$-invariant curves to be discrete elastic rods in general parametrization. However, finding geometric characterizations such as Theorem \ref{thm:elasticrod} is difficult in this general case.

For $n>2$ we obtain a hierarchy of curves. We believe that this hierarchy is a discrete analogue of the hierarchy presented in~\cite{chern2018commuting}. We will explore this hierarchy further in the future.

\section{Cross-ratio systems and constant curvature surfaces}
\label{sec:surfaces}

In this chapter we show how cross-ratio systems and several discrete surface classes of constant curvature can be described as parallelogram nets. These surfaces can even be described as special cross-ratio systems. The simplest surface examples are 2-dimensional parallelogram nets describing K-nets, discrete pseudospherical surfaces in special coordinates. To describe circular nets of constant curvature, however, we have to consider slices in 4-dimensional parallelogram lattices. Remarkably, from this perspective we discover a method to obtain such surfaces from holomorphic data which turns out to be a novel description of the discrete DPW method. The DPW method \cite{dorfmeister1998weierstrass} is a generalized Weierstrass representation for constant mean curvature surfaces and has been discretized in \cite{ogata2017construction,discreteDPW}.

\subsection{Cross-ratio systems}
\label{sec:crsystems}

The cross-ratio of four points $a,b,c,d\in\mathbb{C}$ is defined as
\begin{align*}
cr(a,b,c,d)=\frac{a-b}{b-c}\frac{c-d}{d-a}.
\end{align*}
We consider vertex based maps $f:\mathbb{\Z}^n\to\mathbb{C}$ in the complex plane and their edges $d^i:=f_i-f$. For simplicity, we will assume the four points of each quad to be pairwise distinct.

\begin{definition}
A map $f:\mathbb{\Z}^n\to\mathbb{C}$ is called a cross-ratio system if there exists an edge labelling $\alpha: \mathcal{E}(\mathbb{Z}^n)\to\mathbb{C}$, i.e., $\alpha^i_j=\alpha^i$ for $i\neq j$, such that
\begin{align}
\label{eq:crdd}
cr(f,f_i,f_{ij},f_j)=\frac{d^i d^i_j}{d^j d^j_i}=\frac{(\alpha^i)^2}{(\alpha^j)^2}.
\end{align}
\end{definition}

The edge label requirement \eqref{eq:crdd} is sometimes called factorizing. A cross-ratio system is a map with factorizing cross-ratios.

In a cross-ratio system, we have
\begin{align*}
\frac{(\alpha^j)^2}{d^j}+\frac{(\alpha^i)^2}{d^i_j}=\frac{(\alpha^j)^2}{d^jd^j_i}(d^j_i+d^i)=\frac{(\alpha^j)^2}{d^jd^j_i}(d^i_j+d^j)=\frac{(\alpha^i)^2}{d^i}+\frac{(\alpha^j)^2}{d^j_i}.
\end{align*}
Hence, a map $f^*:\mathbb{Z}^n\to\mathbb{C}$ is defined up to translation by
\begin{align*}
(d^i)^*=f^*_i-f^*=\frac{(\alpha^i)^2}{d^i}.
\end{align*}
This map $f^*$ is called the \emph{dual} of $f$ and it is a cross-ratio system with the same labelling $\alpha$.

Cross-ratio systems can be written as a reduction of parallelogram nets:

\begin{theorem}
\label{thm:crLax}
An edge based map
\begin{align}
\label{eq:crLax}
p:\vec{ \mathcal{E}}(\mathbb{Z}^n)\to\mathbb{C}^{2\times2},\qquad
p^i=\begin{pmatrix}
0 & d^i \\ -\frac{(\alpha^i)^2}{d^i} & 0
\end{pmatrix}
\end{align}
is an evolvable parallelogram net if and only if $d$ is the edge map of a cross-ratio system with labelling $\alpha$.
\end{theorem}

Here, we only consider maps $\alpha$ with $\alpha^i\neq0$ and $\alpha^i\neq\alpha^j$ for $i\neq j$ which for a cross-ratio system already follows from the fact that no two points in a quad coincide.

\begin{proof}
The multiplicative parallelogram equation \eqref{eq:pareqmult} reads
\begin{align*}
(\alpha^i)^2\frac{d^j_i}{d^i}=(\alpha^j)^2\frac{d^i_j}{d^j}\qquad\text{and}\qquad (\alpha^j_i)^2\frac{d^i}{d^j_i}=(\alpha^i_j)^2\frac{d^j}{d^ i_j}
\end{align*}
If $\alpha$ has the labelling property both agree with \eqref{eq:crdd}. Hence, for a cross-ratio system the parallelogram equations are fulfilled and we can calculate evolvability using $f_i\neq f_j$. Conversely, in a parallelogram net the labelling property of $\alpha$ follows from the fact that $\alpha^i:=\sqrt{\det p^i}$ is an invariant of the parallelogram evolution. The additive parallelogram equation \eqref{eq:pareqadd} implies that $d$ can be integrated into a vertex based map and the multiplicative equation implies \eqref{eq:crdd}.
\end{proof}

The Lax representation \eqref{eq:laxpair} is gauge equivalent (with $\lambda=1$) to the original linear Lax representation for cross-ratio systems established in~\cite{nijhoff1997some}.

\begin{remark}
\label{rem:isothermicFactorizingCR}
Theorem~\ref{thm:crLax} can be stated equivalently for isothermic nets, which are characterized by having quads with real and factorizing cross-ratio in $\Im(\mathbb{H})$. The matrices are in $\mathbb{H}^{2\times2}$ and the Lax representation \eqref{eq:laxpair} is gauge equivalent to the Lax representation described in~\cite{bobenko1996discrete}. In \cite{hertrich2000transformations} the representation \eqref{eq:laxpair} appears in the context of deformations of isothermic nets.
\end{remark}

To each cross-ratio system $f$ one can define the corresponding vertex based map $s: \Z^n \to \C$ as follows:
\begin{align}
	\label{eq:sVertexMap}
s(\bm0)=1,\qquad d^i=\alpha^i s_is.
\end{align}
The parallelogram evolution \eqref{eq:parevol} for cross-ratio systems written in terms of $s$ is
\begin{align}
\label{eq:sevol}
\frac{s_{ij}}{s}=\frac{\alpha^js_j-\alpha^is_i}{\alpha^js_i-\alpha^is_j}.
\end{align}

We will now investigate how real surfaces can be found within this class.

\subsection{K-nets}
\label{sec:knets}

K-nets, also known as pseudospherical nets, have originally been studied in \cite{wunderlich1951differenzengeometrie,sauer1950parallelogrammgitter} and have been developed further in~\cite{bobenkoPinkallKNets,discretizationOfSurfacesIntegrable}. They are usually defined as nets in $f:\mathbb{Z}^2\to\mathbb{R}^3\cong\Im(\mathbb{H})$ with planar vertex stars that have the Chebyshev property, which is that opposite edges have equal length. Primitive maps of quaternionic parallelogram nets always have the Chebyshev property, implying that parallelogram nets with planar vertex stars are K-nets.

Here, we present a description of K-nets which is a reduction of both \eqref{eq:crLax} and of the representation given in Theorem \ref{thm:zerofoldedrep}. For this, consider a quaternionic two-dimensional parallelogram net that is completely contained in a plane, let's say the $\bm i,\bm j$-plane. Then, each edge variable $u,v$ has only off diagonal entries and, with $\lambda=e^t$ and $\mu=e^{-t}$, the Lax representation \eqref{eq:laxpair} takes the form
\begin{align}
\label{eq:Knetlax}
U=\begin{pmatrix}
\lambda & \frac{1}{\lambda} a \\ -\frac{1}{\lambda}\bar{a} & \lambda
\end{pmatrix},\quad
V=\begin{pmatrix}
\lambda & \frac{1}{\lambda} b \\ -\frac{1}{\lambda}\bar{b} & \lambda
\end{pmatrix}, \quad \lambda=e^t.
\end{align}
The primitive map $f$ of the net $p=(u,v)$ is a K-net, since lengths are invariants and by definition all edges lie in a common plane. This is also true for all nets in the associated family \eqref{eq:assoedge} since according to Lemma \ref{lem:vertexstars} planar vertex stars stay planar. While the original net is zero-folded and consists of anti-parallelograms, the equally folded nets in the family indeed unfold into non-planar K-nets as exemplified in Fig. \ref{fig:amsler}.

\begin{figure}[h!]
  \centering
  \includegraphics[width=0.8\textwidth]{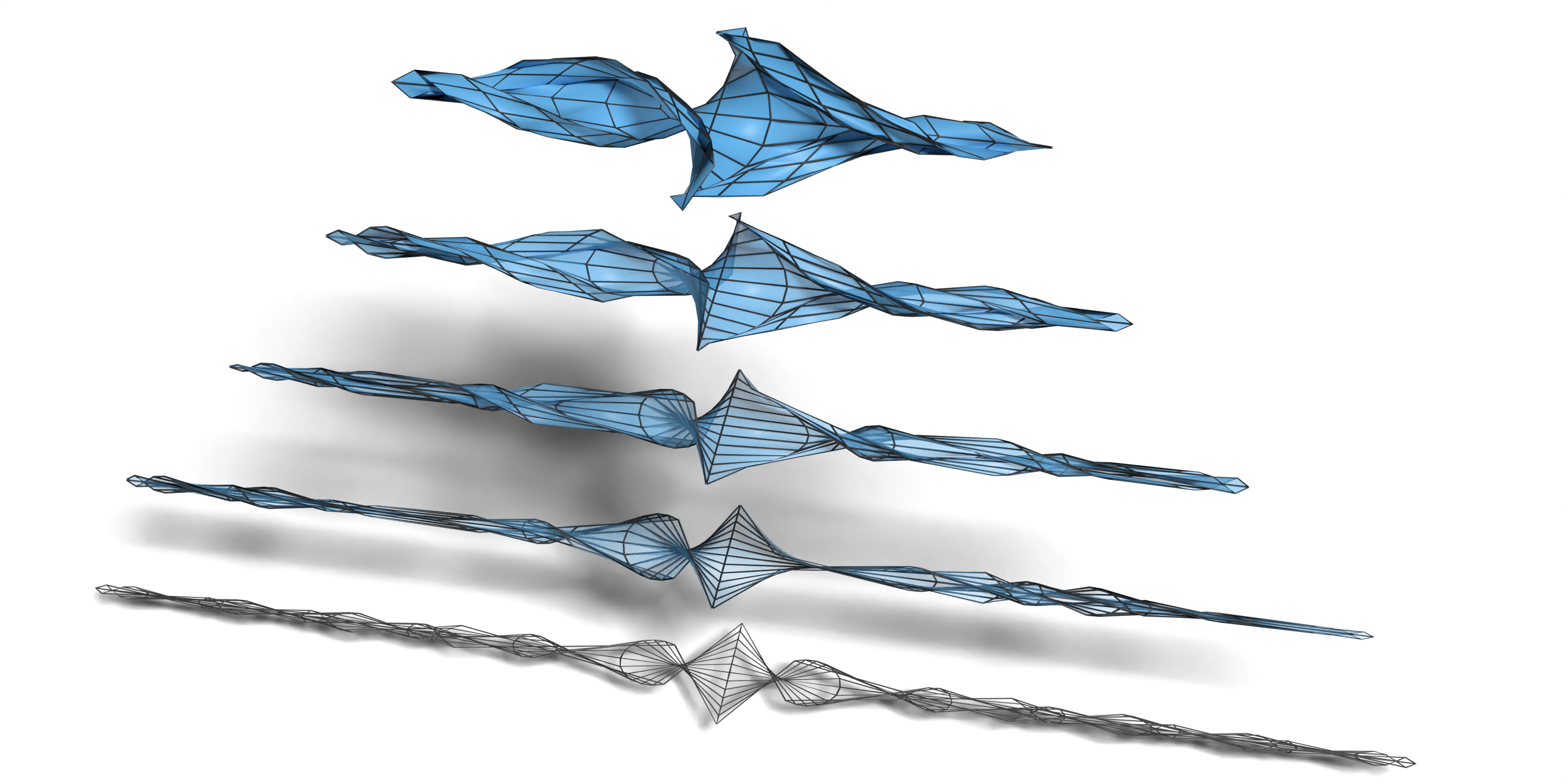}
  \caption{A zero-folded K-net unfolds into its associated family. This example shows an Amsler surface.}
  \label{fig:amsler}
\end{figure}

Also note that the planar parallelogram net $p=(u,v)$ is the special case of \eqref{eq:crLax} where the matrices are quaternions. Indeed, this is the case if and only if $(d^i)^*=\bar{d}^i$ or, equivalently, if the edge labels $\alpha^i$ are real and the vertex based map $s$ from \eqref{eq:sVertexMap} is unitary. Hence, we can consider K-nets as a special case of cross-ratio systems.

We will show that the Lax representation \eqref{eq:Knetlax} is gauge equivalent to the classical K-net Lax representation. This implies that all K-nets can be represented this way and that our associated family coincides with the known associated family.

The classical Lax representation for K-nets given by ~\cite[equation (45),(46)]{discretizationOfSurfacesIntegrable} is of the form
\begin{align}
\label{eq:classicKlax}
L=
\begin{pmatrix}
x & \ci\lambda \\ \ci\lambda & \bar{x}
\end{pmatrix},\quad
M=
\begin{pmatrix}
1 & \frac{y}{\lambda} \\ \frac{-\bar{y}}{\lambda} & 1
\end{pmatrix},\qquad \lambda=e^t.
\end{align}
If $\Phi$ is a frame solving $\Phi_1=L\Phi,\Phi_2=M\Phi$ the K-net is given by the Sym formula $f=2\Im(\Phi^{-1}\Phi')$.

\begin{proposition}
The classical K-net Lax representation \eqref{eq:classicKlax} is gauge equivalent to the Lax representation \eqref{eq:Knetlax}.
\end{proposition}

\begin{proof}
In the Lax representation $L,M$ given by \eqref{eq:classicKlax} we can replace $\lambda=e^t$ with $\lambda^2=e^{2t}$ and rescale $L$ and $M$ with $\lambda$ and $\frac{1}{\lambda}$, respectively. We obtain
\begin{align*}
\hat{L}:=\frac{1}{\lambda}L(\lambda^2)=
\begin{pmatrix}
\frac{x}{\lambda} & \ci\lambda \\ \ci\lambda & \frac{\bar{x}}{\lambda}
\end{pmatrix},\quad
\hat{M}:=\lambda M(\lambda^2)=
\begin{pmatrix}
\lambda & \frac{y}{\lambda} \\ \frac{-\bar{y}}{\lambda} & \lambda
\end{pmatrix}.
\end{align*}
A gauge with $G(k,l)=\begin{pmatrix}0&-\ci\\-\ci&0\end{pmatrix}^k$ gives
\begin{align*}
U=G_1\hat{L}G^{-1}=
\begin{pmatrix}
\lambda & \frac{a}{\lambda} \\ \frac{-\bar{a}}{\lambda} & \lambda
\end{pmatrix},\quad
V=G_2\hat{M}G^{-1}=
\begin{pmatrix}
\lambda & \frac{b}{\lambda} \\ \frac{-\bar{b}}{\lambda} & \lambda
\end{pmatrix},
\end{align*}
where for even $k$ we have $a=-\ci\bar{x},b=y$ and for odd $k$ we have $a=-\ci x,b=-\bar{y}$. This is exactly \eqref{eq:Knetlax}.

None of these transformations affects the compatibility condition, i.e., $L_2M=M_1L$ is equivalent to $U_2V=V_1U$. Also the Sym formulas coincide up to scaling and real part, since $2\Im(\Phi^{-1}\Phi')=\Im(\Psi^{-1}\Psi')$, where $\Phi$ belongs to the pair $L,M$ and $\Psi$ belongs to $U,V$.
\end{proof}

We saw that K-nets could be interpreted as reductions of cross-ratio systems. We will now generalize this concept.

\subsection{Surfaces from 4D cross-ratio systems}
\label{sec:4D}

If one constructs the associated family of a cross-ratio system it will, in general, consist of complex matrices without much additional structure and it is difficult to interpret. If, however, the Lax matrices $P^i$ are quaternions, the associated family consists of quaternions too and describes surfaces in Euclidean three-space. This, as shown in the previous section, is the case if and only if $(d^i)^*=\bar{d}^i$ and we get K-nets throughout the family. We can get more geometric solutions if we don't require \emph{all} points to be quaternions but only, e.g., every second one.

In this section we present two reductions of the cross-ratio system that, at some points, give quaternions in the associated family. Here, the combinatorics are more involved: We consider a 4D cross-ratio system $f:\mathbb{Z}^4\to\mathbb{C}$. The vertices of the form $(i,j,i,j)\in\mathbb{Z}^4$ form a new $\mathbb{Z}^2$, which we call the \emph{diagonal net $D\subset\mathbb{Z}^4$}. Naturally, see Figure~\ref{fig:4comb}, on its edges live the Lax matrices
\begin{align}
L&=(\lambda+\frac{1}{\lambda}p^3_1)(\lambda+\frac{1}{\lambda}p^1), \nonumber \\
M&=(\lambda+\frac{1}{\lambda}p^4_2)(\lambda+\frac{1}{\lambda}p^2). 
\label{eq:4DLMLaxMatrices}
\end{align}

\begin{figure}[h!]
  \centering
  \includegraphics[width=0.68\textwidth]{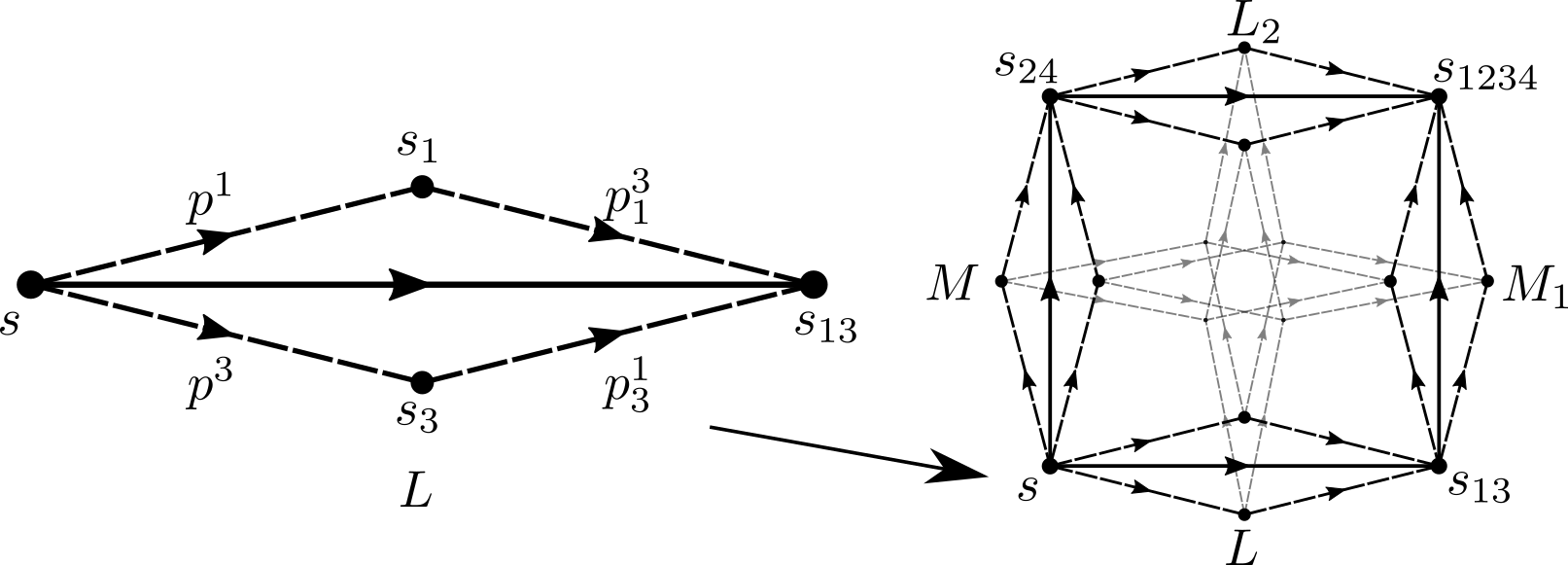}
  \caption{A combinatoric 4D cross-ratio cube with the Lax representation $L,M$ of its diagonal net $D$.}
  \label{fig:4comb}
\end{figure}

\begin{definition}
Let $f:\mathbb{Z}^4\to\mathbb{C}$ be a cross-ratio system with edge labelling $\alpha$, such that the corresponding map $s$ \eqref{eq:sVertexMap} fulfills
\begin{align}
\label{eq:dpwcond}
s_1(x)\bar{s}_3(x)=1,\qquad s_2(x)\bar{s}_4(x)=1
\end{align}
for every $x\in D$. We call $f$ 
\begin{itemize}
\item a \emph{$C^{+}$ lattice} if for all $x\in D$
\begin{align}
\label{eq:K+label}
\alpha^1(x)\bar{\alpha}^3(x)=-1,\qquad \alpha^2(x)\bar{\alpha}^4(x)=-1.
\end{align}

\item a \emph{$C^{-}$ lattice} if for all $x\in D$
\begin{align}
\label{eq:K-label}
\alpha^1(x)=\bar{\alpha}^3(x),\qquad \alpha^2(x)=\bar{\alpha}^4(x).
\end{align}

\end{itemize}
\end{definition}

We will see that these $C^+$ and $C^-$ lattices belong to discrete surfaces of constant positive and negative Gaussian curvature, respectively.

Using \eqref{eq:sevol}, one can verify that for all $x\in D$ we have $s(x)\in\mathbb{R}$ on a $C^{+}$ lattice while we have $s(x)\in\mathbb{S}^1$ on a $C^{-}$ lattice. We demonstrate this for $C^{+}$ and the first diagonal direction: If $s\in\mathbb{R}$, we have
\begin{align}
\label{eq:replacement}
\bar s_{13}=\bar s\frac{\bar\alpha^1\bar s_1-\bar\alpha^3\bar s_3}{\bar \alpha^1\bar s_3-\bar\alpha^3\bar s_1}=s\frac{\frac{-1}{\alpha^3s_3}-\frac{-1}{\alpha^1s_1}}{\frac{-1}{\alpha^3s_1}-\frac{-1}{\alpha^1s_3}}=s\frac{\alpha^1s_1-\alpha^3s_3}{\alpha^1s_3-\alpha^3s_1}=s_{13}.
\end{align}

The $C^{+}$ and $C^{-}$ lattice reductions are consistent:
\begin{proposition}
\label{dpwprop}
Every two-dimensional cross-ratio system can be extended into a unique $C^+$ and a unique $C^-$ lattice.
\end{proposition}

\begin{figure}[h!]
  \centering
  \includegraphics[width=0.3\textwidth]{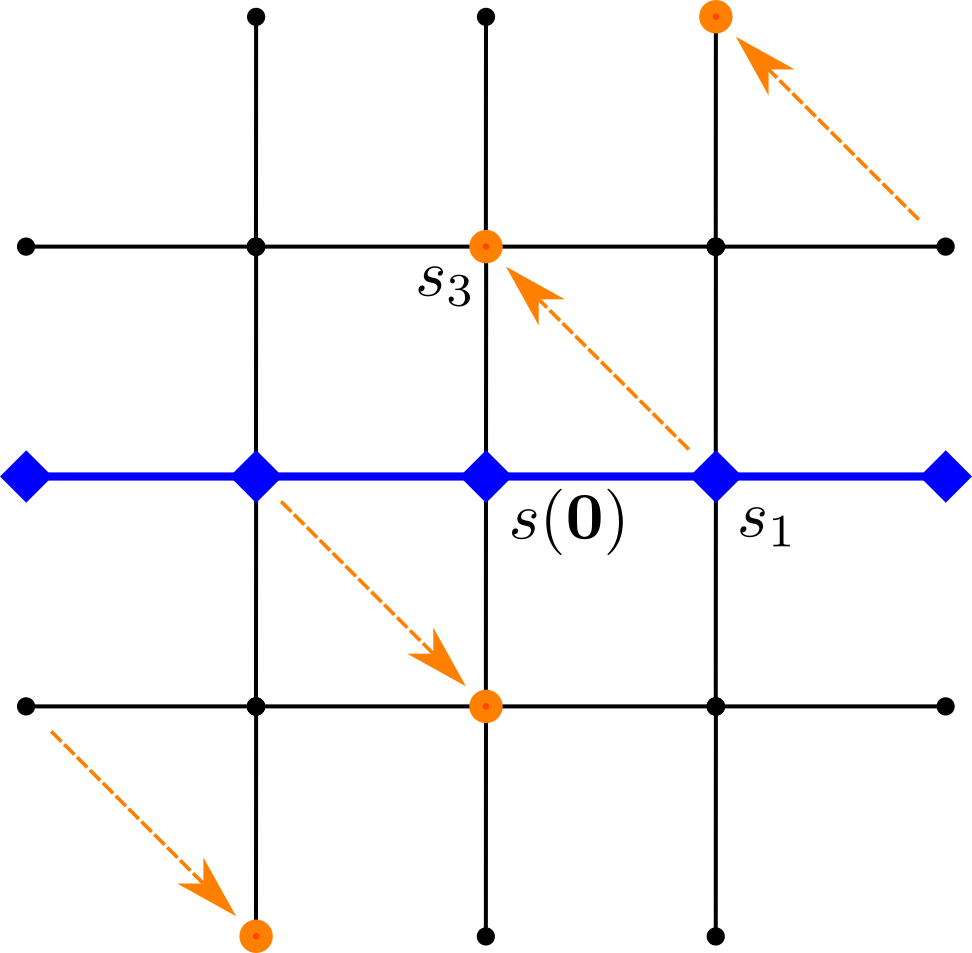}
  \caption{In the 1,3-coordinate plane one can construct the map $s$ only from $s(i,0,0,0)$ (blue) by using \eqref{eq:dpwcond} as indicated by the arrows (orange) and \eqref{eq:sevol} to complete each quad.}
  \label{fig:dpwproof}
\end{figure}

\begin{proof}
Consider a two-dimensional cross-ratio system $f(i,j,0,0)$ in the $1,2$-plane with labelling $\alpha^1,\alpha^2$ and $s$ as in \eqref{eq:sVertexMap}. We will construct a $C^{-}$ lattice from this initial data.

The labelling property of $\alpha$ together with \eqref{eq:K-label} determines $\alpha$ on the whole net. For example, for $\alpha^3$ we have
\begin{align*}
\alpha^3(i,j,k,l)=\alpha^3(k,0,k,0)=\bar\alpha^1(k,0,k,0)=\bar\alpha^1(k,0,0,0).
\end{align*}

Equations \eqref{eq:dpwcond} and \eqref{eq:sevol} determine $s$ on the 1,3-coordinate plane, as seen in Figure~\ref{fig:dpwproof}, and, similarly, on the 2,4-coordinate plane. Then, $s$ is determined on the axes and can be evolved to the whole 4D lattice. Now, the 4D cross-ratio system $f$ is determined by $\alpha$, $s$ and $f(\bm 0)$.

We need to prove consistency of \eqref{eq:dpwcond}. Consider a point $x\in D$ where
\begin{align*}
s\bar{s}=1,\qquad s_3\bar s_1=1,\qquad s_4\bar s_2=1.
\end{align*}
Then, one step of \eqref{eq:sevol}, using the available replacement rules just as in \eqref{eq:replacement}, yields
\begin{align*}
s_{13}\bar s_{13}=s_{34}\bar s_{12}=s_{14}\bar s_{23}=1,
\end{align*}
and a second step yields
\begin{align*}
s_{123}\bar{s}_{134}=s_{124}\bar s_{234}=1.
\end{align*}
Hence, \eqref{eq:dpwcond} is consistent.

The construction of a $C^{+}$ lattice and its consistency proof works exactly the same using \eqref{eq:K+label} in place of \eqref{eq:K-label}. Here, after one step we have
\begin{align*}
s_{13}=\bar s_{13},\qquad s_{34}=\bar s_{12},\qquad s_{14}=\bar s_{23}.
\end{align*}
\end{proof}

Next, we prove that we can indeed get real surfaces in the associated family \eqref{eq:assoedge} for special $\lambda$. For this, we show that $L,M$~\eqref{eq:4DLMLaxMatrices} and their derivatives with respect to $t$ $L',M'$ are quaternions, possibly after a gauge. For $\Phi(\bm0),f_t(\bm 0)\in\mathbb{H}$ this implies that $\Phi(x)$ and $f_t(x)$ are quaternions for all $x\in D$ since $\Phi_{13}=L\Phi$ and
\begin{align*}
(f_t)_{13}-f_t=\Phi_{13}^{-1}\Phi_{13}'-\Phi^{-1}\Phi'=\Phi_{13}^{-1}(L\Phi'+L'\Phi-L\Phi')=\Phi_{13}^{-1}L'\Phi\in\mathbb{H}.
\end{align*}
Here, we only consider the first lattice direction of the diagonal net since both directions work the same. Consequently, we will only investigate $L$ and $L'$. For this, we compute
\begin{align}
\label{eq:explicitL}
L=\begin{pmatrix}\lambda & \frac{\alpha^3s_{13}s_1}{\lambda} \\
\frac{-\alpha^3}{\lambda s_{13}s_1} & \lambda\end{pmatrix}
\begin{pmatrix}\lambda & \frac{\alpha^1s_{1}s}{\lambda} \\
\frac{-\alpha^1}{\lambda s_{1}s} & \lambda\end{pmatrix}
=\begin{pmatrix}
\lambda^2-\frac{1}{\lambda^2}\alpha^3\alpha^1\frac{s_{13}}{s}& s_1(\alpha^1s+\alpha^3s_{13}) \\
-\frac{1}{s_1}(\frac{\alpha^3}{s_{13}}+\frac{\alpha^1}{s}) & \lambda^2-\frac{1}{\lambda^2}\alpha^3\alpha^1\frac{s}{s_{13}}
\end{pmatrix}.
\end{align}
Also, note, that \eqref{eq:sevol} implies
\begin{align}
\label{eq:offdiagonal}
s_1(\alpha^1s+\alpha^3s_{13})=s_3(\alpha^3s+\alpha^1s_{13}).
\end{align}

\begin{proposition}
In a $C^{-}$ lattice with $\lambda=e^t\in\mathbb{R}$ the points $f_t$ of the diagonal net of the associated family are quaternions.
\end{proposition}

\begin{proof}
Here, $L$ as in \eqref{eq:explicitL} is a quaternion. This can be seen using \eqref{eq:dpwcond},\eqref{eq:K-label}, $s,s_{13}\in\mathbb{S}^1$ and, for the off-diagonal entries, \eqref{eq:offdiagonal}.
Also, if $L$ is a quaternion for all $t$, so is its derivative $L'$.
\end{proof}

\begin{proposition}
In a $C^{+}$ lattice with $\lambda=e^{\ci t}\in\mathbb{S}^1$ the points $f_t$ of the diagonal net of the associated family are quaternions.
\end{proposition}

\begin{proof}
A gauge of $L$ as in \eqref{eq:explicitL} with $G:=\begin{pmatrix}1&0\\0&s\end{pmatrix}$ and multiplication with $\frac{1}{\beta}$ defined by $\beta^2:=-\alpha^1\alpha^3\in\mathbb{S}^1$ does not affect the compatibility condition or the associated family. We get
\begin{align}
\label{eq:K+Ltilde}
\tilde{L}:=\frac{1}{\beta}G_{13}LG^{-1}=\begin{pmatrix}
\frac{\lambda^2}{\beta}+\frac{\beta}{\lambda^2}\frac{s_{13}}{s}& \frac{s_1}{\beta s}(\alpha^1s+\alpha^3s_{13}) \\
-\frac{s_{13}}{\beta s_1}(\frac{\alpha^3}{s_{13}}+\frac{\alpha^1}{s}) & \frac{\lambda^2}{\beta}\frac{s_{13}}{s}+\frac{\beta}{\lambda^2}
\end{pmatrix}.
\end{align}
We see that this is a quaternion using $s,s_{13}\in\mathbb{R},\lambda,\beta\in\mathbb{S}^1$ and, for the off-diagonal entries, \eqref{eq:offdiagonal}. Again, its derivative $\tilde L'$ is also a quaternion. 
\end{proof}

In the next two sections, we explore the surfaces obtained from this construction.

\subsection{Circular nets with constant positive Gaussian curvature and their corresponding constant mean curvature nets}

Here, we investigate the connection between parallelogram nets and discrete nets of constant mean curvature (CMC) as in~\cite{discretizationOfSurfacesIntegrable}. CMC nets are \emph{isothermic} which means (as briefly stated in remark~\ref{rem:isothermicFactorizingCR} of section \ref{sec:crsystems}) that they have quads with real factorizing cross-ratio. In particular, this implies that isothermic nets are \emph{circular nets} which means that the four points of each quad are concircular. Isothermic nets, similar to cross-ratio systems, have a dual surface~\cite[equation (36)]{bobenko1996discrete} defined up to translation. If this dual surface can be placed at constant distance to the original surface, the primal and dual net form a pair of CMC nets. At half the distance between these surfaces sits a third surface, which, as in smooth theory, has constant positive Gaussian curvature. Thus, these two types of constant curvature surfaces are described within one system of equations.

The Lax representation for this system as in~\cite{discretizationOfSurfacesIntegrable} is given by the matrices
\begin{align}
\label{eq:cmclax}
L=\begin{pmatrix}
a & \lambda b+\frac{1}{\lambda b} \\ -\frac{\bar b}{\lambda}-\frac{\lambda}{\bar b} & \bar a
\end{pmatrix},\qquad
M=\begin{pmatrix}
d & \lambda e+\frac{1}{\lambda e} \\ -\frac{\bar e}{\lambda}-\frac{\lambda}{\bar e} & \bar d
\end{pmatrix},\quad \lambda=e^{\ci t}.
\end{align}
If they fulfill the compatibility condition $M_1L=L_2M$ they define a frame $\Phi$ by $\Phi_1=L\Phi$, $\Phi_2=M\Phi$ and a family of nets $f_t=\Phi^{-1}\Phi'$ with normals $n_t=\Phi^{-1}\bm k\Phi$. At $t=0$ these are nets of constant positive Gaussian curvature and $f_0\pm\frac{1}{2}n_0$ is a pair of primal and dual CMC nets in isothermic parametrization. For $t\neq0$ they are still nets of constant curvature in different parametrizations.

In~\cite{hoffmann16}, the authors already establish a connection between CMC nets and parallelogram nets: A 3D-cube of a CMC net and its dual form a zero-folded quaternionic parallelogram cube after a combinatoric flip. A more detailed explanation and proof that the classic CMC Lax representation \eqref{eq:cmclax} and the parallelogram net Lax representation \eqref{eq:laxpair} for this cube are gauge equivalent can be found in appendix \ref{sec:cmcCubeGauge}. It also implies that the associated family of CMC coincides with \eqref{eq:assoedge}.

Here, we also establish a different connection to parallelogram nets and, in particular, to the $C^{+}$ lattice.

\begin{theorem}
The Lax representations \eqref{eq:cmclax} for CMC and positive Gaussian curvature nets and \eqref{eq:4DLMLaxMatrices} for the $C^{+}$ lattice are gauge equivalent.
\end{theorem}

This implies that the associated family \eqref{eq:assoedge} of the $C^{+}$ lattice agrees with the known CMC family.

\begin{proof}
Again, we only consider one lattice direction as both work the same.

A gauge of the $C^{+}$ lattice matrix $\tilde{L}$ as in \eqref{eq:K+Ltilde} with $G(i,j,i,j)=\bm j^{i+j}$ and $\lambda^2\to\lambda$ yields that $\sqrt{\frac{s}{s_{13}}}G_{13}\tilde{L}G^{-1}$ is of the form \eqref{eq:cmclax} with
\begin{align}
\label{eq:cmcentries}
b=\frac{1}{\beta\sqrt{\hat s_{13}\hat s}},\quad a=-\frac{\sqrt{\hat s}\hat s_1}{\sqrt{\hat s_{13}}\beta}(\alpha^3\hat s_{13}+\frac{\alpha^1}{\hat s}),\quad \hat{s}:=\begin{cases}s,\quad i+j\text{ even}\\\frac{1}{s},\quad i+j\text{ odd}\end{cases}
\end{align}
Conversely, from a solution to \eqref{eq:cmclax} we can recover $\alpha$ from the fact that $\pm \ci\alpha^1,\pm \ci\alpha^3$ are the roots of $\det L$. They can be chosen such that \eqref{eq:K+label} holds. The values of $s$ can be recovered from \eqref{eq:cmcentries} and yield a $C^{+}$ lattice.
\end{proof}

Therefore, by Proposition~\ref{dpwprop} we can construct CMC nets by uniquely extending a 2D cross-ratio system, in other words, a discrete holomorphic map. Constructing surfaces and harmonic maps from holomorphic maps is usually achieved with the DPW method~\cite{dorfmeister1998weierstrass}. Our method appears to be equivalent to the discrete DPW method for CMC nets given in~\cite{ogata2017construction,discreteDPW}. Understanding the DPW method in terms of a cross-ratio evolution on a 4D lattice yields new insight into this topic. Also, the application of our version of this method seems to be much simpler and more constructive then previous versions. Instead of explicit re-splitting of the matrices we simply evolve the holomorphic map directly into a 4D cross-ratio system.

\subsection{Circular nets with constant negative Gaussian curvature}
\label{sec:cknets}

Discrete circular nets of constant negative Gaussian curvature (cK-nets) have been introduced in~\cite{cKnets} as circular nets which are diagonal nets of 4D K-nets.

A Lax representation is given by products of the K-net Lax matrices \eqref{eq:classicKlax} and can be written as
\begin{align}
\label{eq:cklax}
L&=\begin{pmatrix}
\frac{1}{t_{(1)}}\frac{l}{r}+t_{(1)}lr_1 & \ci(\lambda-\frac{rr_1}{\lambda}) \\
\ci(\lambda-\frac{1}{\lambda rr_1}) & \frac{1}{t_{(1)}}\frac{r}{l}+t_{(1)}\frac{1}{lr_1}
\end{pmatrix},\nonumber\\
M&=\begin{pmatrix}
\frac{1}{t_{(2)}}\frac{m}{r}+t_{(2)}mr_2 & \ci(\lambda-\frac{rr_2}{\lambda}) \\
\ci(\lambda-\frac{1}{\lambda rr_2}) & \frac{1}{t_{(2)}}\frac{r}{m}+t_{(2)}\frac{1}{mr_2}
\end{pmatrix},
\end{align}
where $r\in\mathbb{S}^1$ is a function on vertices, $t_{(1)},t_{(2)}\in\mathbb{R}$ are functions on edges with have the labelling property and $l,m\in\mathbb{C}$ are functions on edges. If $L$ and $M$ fulfill $M_1L=L_2M$ they define a frame $\Phi$ by $\Phi_1=L\Phi$, $\Phi_2=M\Phi$. With $\lambda=e^{t}$ this defines a family of nets by $f_t=\Phi^{-1}\Phi'$. Then, $f_0$ is a cK-net and the members of its family still have constant negative Gaussian curvature.

As multidimensional K-nets they are described in terms of parallelogram nets. Since these multidimensional nets have quaternions along all edges, they do not belong to the $C^-$ lattices discussed in section \ref{sec:4D}.

However, in~\cite{cKnets} it is also shown that the description \eqref{eq:cklax} can be extended to the case where $t_{(i)}$ is not real but unitary. Then, $f_t$ still constitutes a family of constant negative curvature nets. This description does not factorize into a K-net anymore. Instead, it factorizes into a $C^{-}$ lattice:

\begin{theorem}
The Lax representations \eqref{eq:cklax} for constant negative Gaussian curvature nets for unitary $t_{(i)}$ and \eqref{eq:4DLMLaxMatrices} for the $C^{-}$ lattice for unitary $\alpha^i$ are gauge equivalent.
\end{theorem}

Again, this implies that the associated family \eqref{eq:assoedge} of the $C^{-}$ lattice agrees with the cK-net family.

\begin{proof}
We only consider one lattice direction.

Gauging the $C^{-}$ Lax matrix $L$ \eqref{eq:explicitL} with $G(i,j,i,j)=\bm i^{i+j}$ and $\lambda^2\to\lambda$ yields that $G_{13}LG^{-1}$ is of the form \eqref{eq:cklax} with
\begin{align*}
t_{(1)}=\frac{1}{\alpha^1}=\alpha^3,\quad r=\begin{cases}s, &\text{for even }i+j,\\
\frac1s, &\text{for odd }i+j\end{cases},\qquad
l=\begin{cases}\frac{1}{is_1}, &\text{for even }i+j,\\
is_1 &\text{for odd }i+j\end{cases}.
\end{align*}
For the converse, we can apply the same gauge backwards.
\end{proof}

In~\cite{cKnets} these matrices for unitary $t_{(i)}$ are used to construct special transformations of surfaces, so called \emph{breather} solutions. Proposition \ref{dpwprop} allows us to construct lattices with these edges from given 2D cross ratio systems. This means that we can construct lattices of surfaces where neighboring surfaces are related by breather transformations. Figure~\ref{fig:breather} (left) shows a lattice of transformations of the straight line generated by the square grid as discrete holomorphic map.

\begin{figure}[h!]
  \center
  \includegraphics[width=0.65\textwidth]{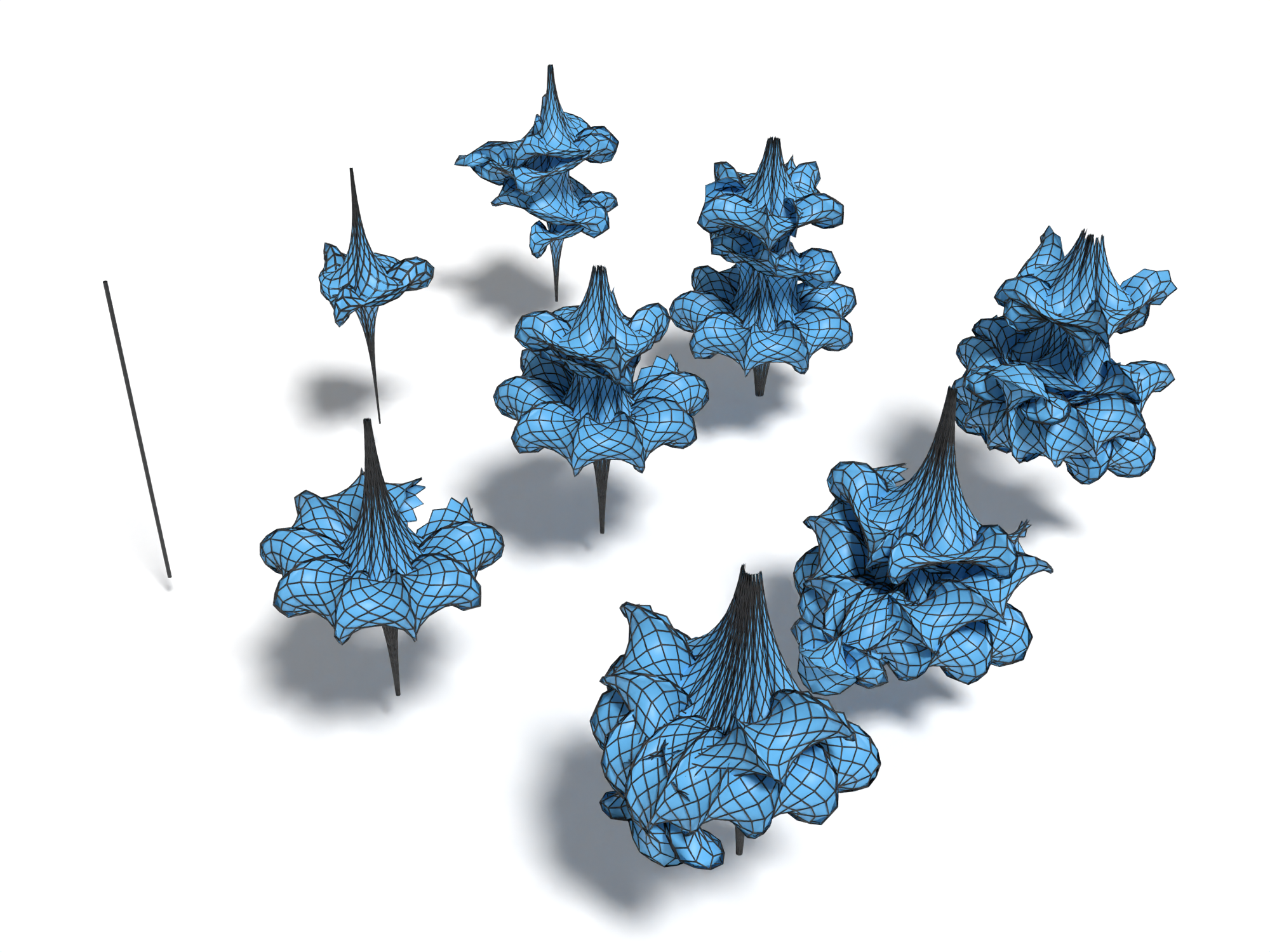}
  \hspace{10mm}
  \includegraphics[width=0.24\textwidth]{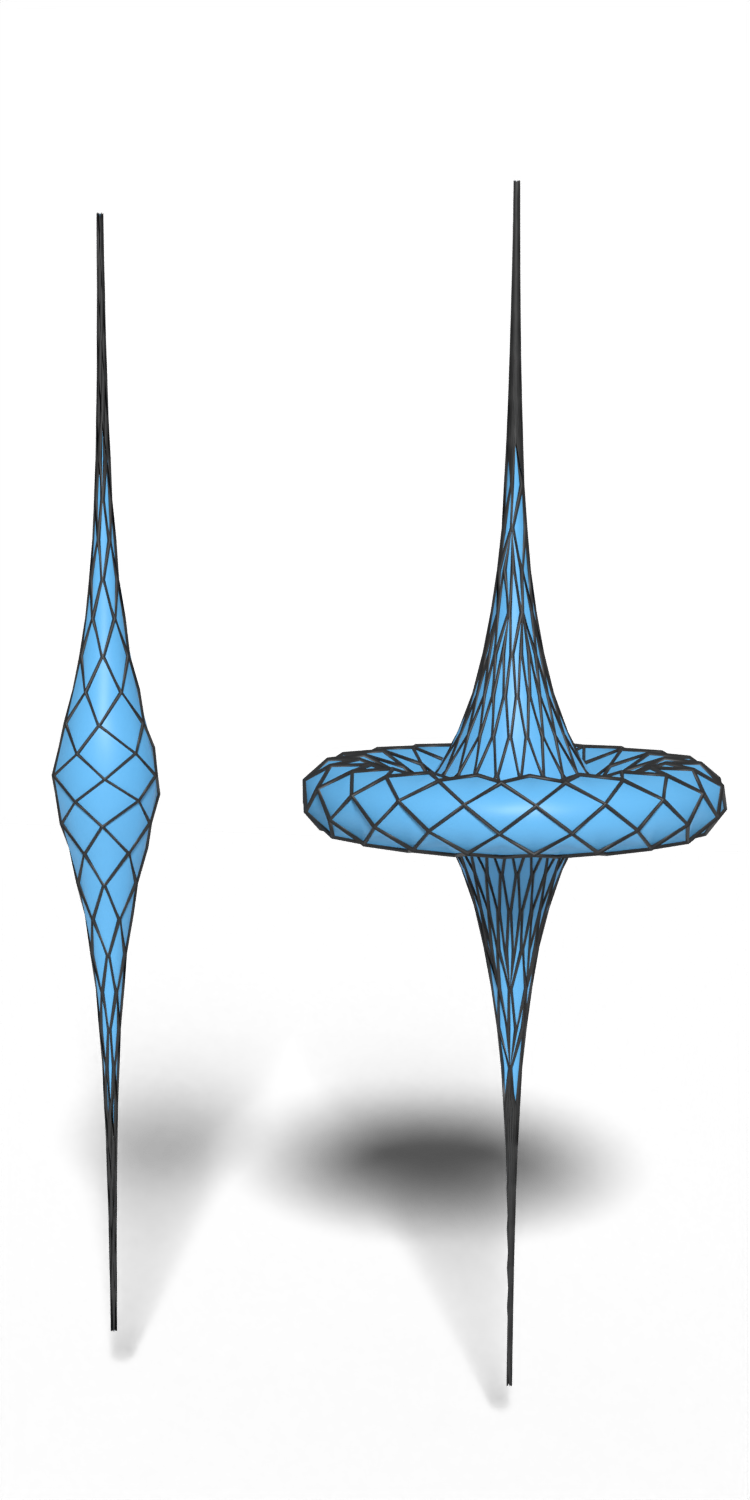}
  \caption{Left: A lattice of breather transformations. Right: Factors of a breather transformation.}
  \label{fig:breather}
\end{figure}

Finally, note that in the $C^{-}$ lattice $L$ as in \eqref{eq:explicitL} is a quaternionic polynomial in the real parameter $\lambda$. Besides the factorization into cross-ratio edges presented here, we know from Theorem \ref{thm:facquat} that there must also be a quaternionic factorization of $L$. We saw the phenomenon of these different factorizations in Example \ref{ex:twoFacs}. An implication is that between a surface and its breather transformation there exist two real surfaces belonging to the two possible quaternionic factorizations. These surfaces are parallelogram nets and their geometric properties are still being investigated. As far as we know, this is the first time one has been able to factorize a breather transformation into real solutions, albeit with different surface geometry. In figure \ref{fig:breather} (right) we see the two real surfaces that combinatorially lie between the straight line and a breather transform.

\section{Other algebras and the Moutard equation}
\label{sec:tnets}

The examples studied in detail throughout this paper belong to the quaternions $\mathbb{H}$ or to $\mathbb{C}^{2\times2}$. We briefly mentioned examples in other algebras like isothermic surfaces in $\mathbb{H}^{2\times2}$~\cite{hertrich2000transformations,bobenko1996discrete}, motions of linkages in terms of dual quaternions~\cite{hegedus2013factorization} or the linear Lax representation of affine spheres~\cite{discreteAffineSpheres} in $\mathbb{R}^{3\times3}$.
Often, geometry outside the quaternions can be described by some sort of algebraic anti-parallelogram. We will make this explicit for Moutard nets in quadrics, which describe various geometric systems such as isothermic nets~\cite{isothermicInSphereGeometriesMoutard,doliwa2007generalized,schief2001isothermic}. We describe these nets in terms of parallelogram nets in a Clifford algebra and explain how one can obtain a geometric associated family.

We consider nets in the space $\mathbb{R}^{p,q}$ which is $\mathbb{R}^{p+q}$ equipped with the inner product
\begin{align*}
\langle x,y\rangle=x_1y_1+\cdots +x_p y_p-x_{p+1}y_{p+1}-\cdots-x_{p+q}y_{p+q}.
\end{align*}
A net $f:\mathbb{Z}^n\to\mathbb{R}^{p,q}$ is in a quadric if $\langle f,f\rangle=\kappa$ for some global constant $\kappa\in\mathbb{R}$.

\begin{definition}
A map $f:\mathbb{Z}^n\to\mathbb{R}^{p,q}$ is called a Moutard net if its quadrilaterals fulfill a Moutard equation, i.e., if there exists $a:\{\text{faces of }\mathbb{Z}^n\}\to\mathbb{R}$, such that for all $1\leq i< j\leq n$
\begin{align}
\label{eq:moutard}
f_{ij}-f=a_{ij}(f_j-f_i).
\end{align}
\end{definition}
This just means that the diagonals of each quadrilateral are parallel. For simplicity, we will only consider nets where neither diagonal vanishes.

Now, consider a Moutard net in a quadric. Plugging \eqref{eq:moutard} into $\kappa=\langle f_{ij},f_{ij}\rangle$ yields
\begin{align}
\label{eq:moutardAij}
a_{ij}(\kappa-\langle f_i,f_j\rangle)=\langle f,f_i-f_j\rangle.
\end{align}
Hence, generically, the fourth point $f_{ij}$ of every quad is determined by the other three points by
\begin{align}
\label{eq:moutardquadric}
f_{ij}=f+\frac{\langle f,f_i-f_j\rangle}{\kappa-\langle f_i,f_j\rangle}(f_j-f_i).
\end{align}
One can verify that $\langle f,f_i\rangle$ has the labelling property and we indeed have anti-parallelograms.

For an algebraic description we turn to the clifford algebra $Cl_{p,q}(\mathbb{R})$ which is the algebra generated from the standard basis $e_1,...,e_{p+q}$ and the following relations:
\begin{align*}
e_ie_j=-e_je_i,\qquad e_i^2=-\langle e_i,e_i\rangle.
\end{align*}
Clifford algebras are graded algebras. For example, scalars have grade zero and elements of the vector space $\mathbb{R}^{p,q}$ have grade one. Two such vectors $a,b\in \mathbb{R}^{p,q}$ fulfill
\begin{align*}
-a^2=\langle a,a\rangle,\quad ab+ba=-2\langle a,b\rangle.
\end{align*}

\begin{lemma}
A net $f$ in a quadric is a Moutard net if and only if
\begin{align}
\label{eq:moutardClifford}
(f_{ij}-f)(f_j-f_i)=2\langle f,f_j-f_i\rangle.
\end{align}
\end{lemma}

\begin{proof}
If $f$ is a Moutard net we get this equation from \eqref{eq:moutard}, \eqref{eq:moutardAij} and $-(f_j-f_i)^2=2(\kappa-\langle f_i,f_j\rangle)$.
Conversely, if the product of the diagonals of $f$ is a scalar the diagonals are parallel and $f$ is a Moutard net.
\end{proof}

Now, we describe the Moutard net as a parallelogram net.

\begin{proposition}
A net $f$ in a quadric of $\mathbb{R}^{p,q}$ is a Moutard net if and only if the edge based net $p$ defined by $p^i=f_i-f$ is a parallelogram net in $Cl_{p,q}(\mathbb{R})$.
\end{proposition}

\begin{proof}
While the additive equation \eqref{eq:pareqadd} holds by construction, the multiplicative equation \eqref{eq:pareqmult} reads
\begin{align*}
0&=p^i_jp^j-p^j_ip^i=(f_{ij}-f_j)(f_j-f)-(f_{ij}-f_i)(f_i-f)\\
&=f_{ij}(f_j-f_i)+(f_j-f_i)f=(f_{ij}-f)(f_j-f_i)-2\langle f,f_j-f_i\rangle
\end{align*}
and is equivalent to \eqref{eq:moutardClifford}. Here, we used $f_i^2=f_j^2=-\kappa$.
\end{proof}

The Lipschitz group is the group of products $a=a_1\cdots a_k\in Cl_{p,q}(\mathbb{R})$ of invertible vectors $a_i\in\mathbb{R}^{p,q}$. These are exactly the elements of the algebra which preserve the grade of a vector $b$ under the conjugation $b\mapsto aba^{-1}$. Any element of the Lipschitz group is either graded even or odd, depending on the number of vectors in the product.

Since $p^i=f_i-f$ is of grade one the Lax representation $\lambda+\mu p^i$ is not in the Lipschitz group. The associated family \eqref{eq:assoedge} does not preserve the grade and $p^i_t$ is not in the vector space which makes geometric interpretation difficult. However, we can also assign a parallelogram net to a Moutard net where $p^i$ is a bi-vector, i.e., an element of grade two.

\begin{proposition}
\label{prop:prodLax}
A net $f$ in a quadric of $\mathbb{R}^{p,q}$ is a Moutard net if and only if the edge based net $p$ defined by $p^i=f_if$ is a parallelogram net in $Cl_{p,q}(\mathbb{R})$.
\end{proposition}

If $\kappa\neq0$ every $f$ is invertible and $f$ can be seen as primitive map of $p$ in the sense of Remark \ref{rem:multPrimitiveMap}.

\begin{proof}
The multiplicative equation \eqref{eq:pareqmult} always holds, since
\begin{align*}
p^j_ip^i&=f_{ij}f_if_if=-\kappa f_{ij}f=f_{ij}f_jf_jf=p^i_jp^j
\end{align*}
For the additive equation \eqref{eq:pareqadd} consider
\begin{align*}
p^i_j+p^j&=f_{ij}f_j+f_jf=(f_{ij}-f)f_j-2\langle f,f_j\rangle\\
p^j_i+p^i&=f_{ij}f_i+f_if=(f_{ij}-f)f_i-2\langle f,f_i\rangle.
\end{align*}
Hence, \eqref{eq:pareqadd} is equivalent to \eqref{eq:moutardClifford}.
\end{proof}

Now, the Lax representation $\lambda+\mu p^i$ is a bi-vector in the Lipschitz group. We can consider the edge values $p^i_t$ in the associated family \eqref{eq:assoedge}. The frame $\Phi$ is in the Lipschitz group group and, hence, conjugation with $\Phi$ preserves the grade. Then, $p^i_t=\Phi^{-1}(P^i)^{-1}(P^i)'\Phi$ is indeed a bi-vector again. As corresponding vertex based net we define the vector valued map $f_t:=\Phi^{-1}f\Phi$.

\begin{proposition}
Let $f$ be a Moutard net in a quadric and let $p$ be the parallelogram net defined by $p^i=f_if$ with a frame $\Phi$. Then, the family of nets defined by $f_t:=\Phi^{-1}f\Phi$ is a family of Moutard nets in the same quadric as $f$.
\end{proposition}

\begin{proof}
Since $\Phi$ is in the Lipschitz group $f_t$ is of grade one, i.e., in the vector space $\mathbb{R}^{p,q}$. Also it is in the same quadric as $f$.
We will show that the edge based net $q_t$ defined by $q^i_t=(f_t)_if_t$ is a parallelogram net. Then, by Proposition \ref{prop:prodLax}, $f_t$ is a Moutard net. Specifically, it will turn out that $q_t$ is of the form $q^i_t=r+sp^i_t$ where $p_t$ is in the associated family of $p$ and $s,r\in\mathbb{R}$ are globally constant. Then, $q_t$ is primarily equivalent to $p_t$ and, hence, also a parallelogram net.

We calculate
\begin{align*}
q^i_t=(f_t)_if_t=\Phi_i^{-1}f_i\Phi_i\Phi^{-1}f\Phi=\Phi_i^{-1}f_i(\lambda+\mu f_if)f\Phi=\Phi_i^{-1}(\mu\kappa^2+\lambda f_if)\Phi
\end{align*}
and
\begin{align*}
r+sp^i_t=r\Phi_i^{-1}P^i\Phi+s\Phi^{-1}_i(P^i)'\Phi=\Phi_i^{-1}((r\lambda+s\lambda')+(r\mu+s\mu')f_if)\Phi.
\end{align*}
These coincide if we choose $r=\frac{\lambda\lambda'-\mu\mu'\kappa^2}{\lambda'\mu-\lambda\mu'}$ and $s=\frac{\mu^2\kappa^2-\lambda^2}{\lambda'\mu-\lambda\mu'}$ . Note, that $r$ and $s$ are globally constant, but depend on the parameter $t$.
\end{proof}

We found a geometric associated family. Its application and geometric investigation will be the subject of future research.

\appendix

\section{Gauge equivalence of a CMC cube}
\label{sec:cmcCubeGauge}

To understand the connection between CMC nets and quaternionic parallelogram nets given in~\cite{hoffmann16}, consider a two-dimensional net $f:\mathbb{Z}^2\to \Im(\mathbb{H})$ combined with a second net $f_3:=f^*$ to be a two-layered three-dimensional net. From such a two-layered net, we can obtain a second two-layered net $g$ by a combinatoric flip: $g=f$ and $g_3=f_3$ if $k+l$ is even and $g=f_3$ and $g_3=f$ if $k+l$ is odd. Both nets consist of the same points but the edges of one net are the diagonals of the other. In Theorem 3.33 of~\cite{hoffmann16} we find a connection between isothermic CMC nets and parallelogram nets:

\begin{theorem}
If $f$ is an isothermic CMC net with dual $f^*$ the corresponding net $g$ is a zero-folded quaternionic parallelogram net.
\end{theorem}

Note, that they even show that all CMC nets in the family yield equally folded parallelogram nets. However, since their notion of equally folded is slightly different from ours (see Remark \ref{rem:foldingparam}) we will only use the special case of isothermic and zero-folded nets.

\begin{figure}[h!]
  \centering
  \includegraphics[width=0.38\textwidth]{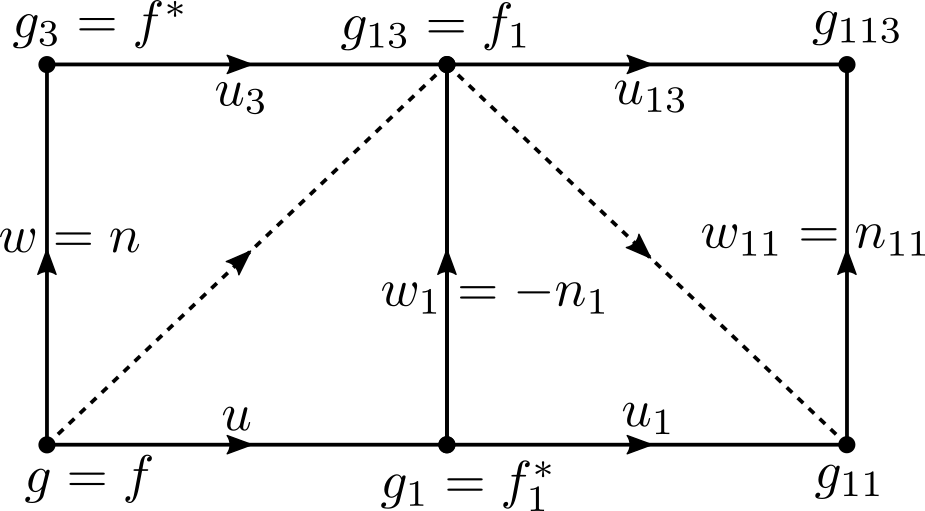}
  \hspace{7mm}
  \includegraphics[width=0.26\textwidth]{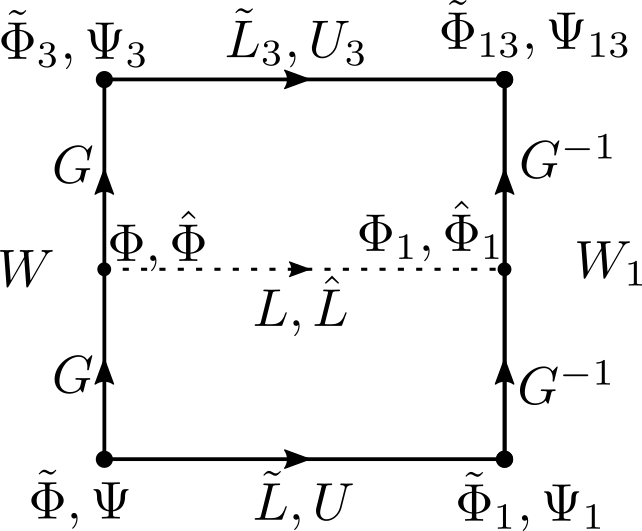}
  \caption{Left: Combinatorics of the two-layered parallelogram net $g$ and the corresponding CMC net $f$ at even $k+l$. Right: Combinatorics at even $k+l$ of the various frames and Lax matrices used in the proof.}
  \label{fig:flipcubes}
\end{figure}

Consider the Lax representation \eqref{eq:laxpair} (with $\lambda (t)=\cos (t),\mu (t)=\sin (t)$) for a zero-folded two-layered three-dimensional parallelogram net $p=(u,v,w)$ 
\begin{align}
\label{eq:cmcCubeLax}
U=\cos (t)+\sin (t)\, u,\quad V=\cos (t)+\sin (t)\, v,\quad W=\cos (t)+\sin (t)\, w.
\end{align}
From Theorems \ref{thm:zerofoldedrep} and \ref{thm:assofamily2} we know that nets in its associated family \eqref{eq:assoedge} are primarily equivalent to equally folded nets. We will assume the net to be scaled such that $w\in\mathbb{S}^2$.

\begin{theorem}
The representation \eqref{eq:cmcCubeLax} and the classic representation for CMC \eqref{eq:cmclax} are gauge equivalent.
\end{theorem}

Again, this implies that our associated family agrees with the CMC associated family for this reduction of parallelogram nets.

\begin{proof}
First, consider a frame $\Phi$ defined by the Lax matrices \eqref{eq:cmclax} with $\lambda=e^{it}$. From it, we also get the CMC net $f$ with normal map $n$ and dual $f^*$. The corresponding two-layered net $g$ has the CMC net as diagonals (see Fig. \ref{fig:flipcubes}). The net $g$ has edges $u,v,w\in \Im(\mathbb{H})$, where $w=(-1)^{k+l}n$. We introduce a two-layered frame $\tilde{\Phi}$ by
\begin{align*}
\tilde{\Phi}=\begin{cases}G^{-1}\Phi, & \text{if $k+l$ is even}\\
G\Phi, &  \text{if $k+l$ is odd}\end{cases},\qquad
\tilde{\Phi}_3=\begin{cases}G\Phi, & \text{if $k+l$ is even}\\
G^{-1}\Phi, &  \text{if $k+l$ is odd}\end{cases},
\end{align*}
where $G=\begin{pmatrix}\sqrt\lambda & 0 \\ 0 & \frac{1}{\sqrt\lambda}\end{pmatrix}=\begin{pmatrix}e^{\frac{\ci t}{2}} & 0 \\ 0 & e^{-\frac{\ci t}{2}}\end{pmatrix}$. Now, it is easy to check the equality $g=\tilde{\Phi}^{-1}\tilde{\Phi}'$ and, hence, $\tilde{\Phi}$ is a frame for $g$. At even $k+l$, we can calculate
\begin{align*}
\tilde{L}&=\tilde{\Phi}_1\tilde{\Phi}^{-1}=GLG=\begin{pmatrix}
\lambda a & \lambda b+\frac{1}{\lambda b} \\ -\frac{\bar b}{\lambda}-\frac{\lambda}{\bar b} & \frac{\bar a}{\lambda}
\end{pmatrix}\\
&=\frac{\lambda+\frac{1}{\lambda}}{2}\begin{pmatrix}
a & b+\frac{1}{b} \\ -\bar b-\frac{1}{\bar b} & \bar a
\end{pmatrix}+\frac{\lambda-\frac{1}{\lambda}}{2\ci}\begin{pmatrix}
\ci a & \ci b-\frac{\ci}{b} \\ \ci\bar b-\frac{\ci}{\bar b} & -\ci\bar a
\end{pmatrix}
=\cos(t)\tilde{L}(1)+\sin(t)\tilde{L}(\ci).
\end{align*}
With similar calculations we can show
\begin{align*}
\tilde{L}=\cos(t)\tilde{L}(1)+\sin(t)\tilde{L}(\ci),\quad \tilde{M}=\cos(t)\tilde{M}(1)+\sin(t)\tilde{M}(\ci)
\end{align*}
at all points of both layers. Also, for the third direction, we have
\begin{align*}
\tilde{\Phi}_3\tilde{\Phi}^{-1}=\begin{cases}G^2&=\cos(t)+\sin(t)\bm k,\quad \text{if $k+l$ is even}\\
G^{-2}&=\cos(t)-\sin(t)\bm k,\quad  \text{if $k+l$ is odd}\end{cases}.
\end{align*}
We can normalize the coefficient of $\cos(t)$ of $\tilde{L},\tilde{M}$ with the gauge $\Psi:=\tilde{\Phi}(1)^{-1}\tilde{\Phi}$:
\begin{align*}
U&:=\Psi_1\Psi^{-1}=\cos(t)+\sin(t)\tilde{\Phi}(1)^{-1}_1\tilde{L}(\ci)\tilde{\Phi}(1)=\cos(t)+\sin(t)u,\\
V&:=\Psi_2\Psi^{-1}=\cos(t)+\sin(t)\tilde{\Phi}(1)^{-1}_2\tilde{M}(\ci)\tilde{\Phi}(1)=\cos(t)+\sin(t)v,\\
W&:=\Psi_3\hat\Psi^{-1}=\cos(t)+(-1)^{k+l}\sin(t)\Phi(1)^{-1}\bm k\Phi(1)=\cos(t)+\sin(t)w.
\end{align*}
The last equality for $U$ and $V$ follows from the fact that the Sym formula at $t=0$ applied to $\Psi$ gives $g$ which has $u,v$ as edges, but we can also directly compute the Sym formula for a Lax matrix of the form $\cos(t)+\sin(t)x$ to give $x$ and, thus, $x$ needs to agree with $u$ (or $v$).

Hence, we have shown that $g$ is a two-layered zero-folded parallelogram net and that its Lax representation can be directly obtained from the CMC Lax matrices.

For the converse, consider a two-layered zero-folded parallelogram net $g$, such that $n:=(-1)^{k+l}w\in\mathbb{S}^2$ with a two-layered frame $\Psi$ given by the Lax representation $U,V,W$. The gauge $\tilde{\Phi}:=(\bm k+n)\Psi$ yields
\begin{align*}
\tilde{\Phi}_3\tilde{\Phi}^{-1}&=(\bm k+n)(\cos(t)+(-1)^{k+l}\sin(t)n)(\bm k+n)^{-1}\\
&=\cos(t)+(-1)^{k+l}\sin(t)\bm k=\begin{cases}G^2, & \text{if $k+l$ is even}\\
G^{-2}, &  \text{if $k+l$ is odd}\end{cases},
\end{align*}
with $G$ as above. Also the $\lambda$-dependency of $\tilde{L}$ can be expressed as
\begin{align*}
\tilde{L}=\tilde{\Phi}_1\tilde{\Phi}^{-1}=(\bm k+n_1)(\frac{\lambda+\frac{1}{\lambda}}{2}+\frac{\lambda-\frac{1}{\lambda}}{2\ci}u)(\bm k+n)^{-1}=:\begin{pmatrix}
\lambda\hat a+\frac{1}{\lambda}\hat d & \lambda \hat b+\frac{1}{\lambda}\hat c \\
-\frac{1}{\lambda}\bar{\hat{b}}-\lambda\bar{\hat{c}} & \frac{1}{\lambda}\bar{\hat{a}}+\lambda\bar{\hat{d}}\end{pmatrix}.
\end{align*}
At even $k+l$ the Lax matrix assigned to the diagonal $u+w_1$ is the product
\begin{align*}
G^{-2}\tilde{L}=\begin{pmatrix}
\hat a+\frac{1}{\lambda^2}\hat d &  \hat b+\frac{1}{\lambda^2}\hat c \\
-\bar{\hat{b}}-\lambda^2\bar{\hat{c}} & \bar{\hat{a}}+\lambda^2\bar{\hat{d}}\end{pmatrix}.
\end{align*}
At $t=0$ we have $G=1$ and we can use the Sym formula to compute
\begin{align*}
\tilde{\Phi}_{1}(G^{-2}\tilde{L})'\bm k\tilde{\Phi}&=\tilde{\Phi}_{13}(G^{-2}\tilde{L})'\tilde{\Phi}^{-1}\tilde{\Phi}\bm k\tilde{\Phi}=(u_2+w)w\\
&=w_1(u+w_1)=-\tilde{\Phi}_{1}\bm k\tilde{\Phi}_{1}^{-1}\tilde{\Phi}_{13}(G^{-2}\tilde{L})'\tilde{\Phi}=-\tilde{\Phi}_{1}\bm k(G^{-2}\tilde{L})'\tilde{\Phi}.
\end{align*}
Now, at $t=0$, we have $(G^{-2}\tilde{L})'\bm k=-\bm k(G^{-2}\tilde{L})'$ which implies that $(G^{-2}\tilde{L})'$ is purely off-diagonal, i.e., $\hat{d}=0$.

Now we can introduce the one-layered frame $\hat{\Phi}:=G\tilde{\Phi}$ for even $k+l$ and $\hat{\Phi}:=G^{-1}\tilde{\Phi}$ for odd $k+l$. This yields Lax matrices of the form
\begin{align*}
\hat{L}=\hat{\Phi}_1\hat{\Phi}^{-1}=\begin{pmatrix}
\hat a & \lambda \hat b+\frac{1}{\lambda}\hat c \\
-\frac{1}{\lambda}\bar{\hat{b}}-\lambda\bar{\hat{c}} & \bar{\hat{a}}\end{pmatrix},\quad
\hat{M}=\begin{pmatrix}
\hat d & \lambda \hat e+\frac{1}{\lambda}\hat f \\
-\frac{1}{\lambda}\bar{\hat{e}}-\lambda\bar{\hat{f}} & \bar{\hat{d}}\end{pmatrix}.
\end{align*}
The $\lambda^2$ coefficient of $\hat{L}_2\hat{M}=\hat{M}_1\hat{L}$ yields $\hat b_2\bar{\hat f}=\hat e_1\bar{\hat c}$ and $\bar{\hat c}_2\hat e=\bar{\hat f}_1\hat b$ which implies that a vertex variable $x$ with $x_1=\frac{\bar{x}}{\hat b\hat c},x_2=\frac{\bar{x}}{\hat e\hat f}$ is well-defined. A final gauge $\Phi:=\begin{pmatrix}\sqrt{x} & 0 \\ 0 & \sqrt{\bar{x}}\end{pmatrix}\hat\Phi$ yields
\begin{align*}
L=\Phi_1\Phi^{-1}=\begin{pmatrix}\sqrt{x_1} & 0 \\ 0 & \sqrt{\bar{x}_1}\end{pmatrix}\hat{L}\begin{pmatrix}\frac{1}{\sqrt{x}} & 0 \\ 0 & \frac{1}{\sqrt{\bar{x}}}\end{pmatrix}=\begin{pmatrix}
a & \lambda  b+\frac{1}{\lambda}\frac{1}{b} \\
-\frac{1}{\lambda}\bar{b}-\lambda\frac{1}{\bar{b}} & \bar{a}\end{pmatrix},
\end{align*}
where $a:=\frac{\sqrt{x_1}}{\sqrt{x}}\hat{a}$, $b:=\frac{\sqrt{\hat{b}}}{\sqrt{\hat{c}}}$. The same calculation can be done for $M=\Phi_2\Phi^{-1}$ and $L,M$ have exactly the form \eqref{eq:cmclax}.
\end{proof}

\footnotesize
\bibliographystyle{plain}
\bibliography{references}

\begin{thebibliography}{10}

\bibitem{adler1993recuttings}
Vsevolod~E. Adler.
\newblock Recuttings of polygons.
\newblock {\em Functional Analysis and Its Applications}, 27(2):141--143, 1993.

\bibitem{affolter2023integrable}
Niklas Affolter, Terrence George, and Sanjay Ramassamy.
\newblock Integrable dynamics in projective geometry via dimers and triple
  crossing diagram maps on the cylinder.
\newblock {\em arXiv:2108.12692}, 2023.

\bibitem{arnold2022cross}
Maxim Arnold, Dmitry Fuchs, Ivan Izmestiev, and Serge Tabachnikov.
\newblock Cross-ratio dynamics on ideal polygons.
\newblock {\em International Mathematics Research Notices}, 2022(9):6770--6853,
  2022.

\bibitem{bergou2008elasticrods}
Mikl{\'o}s Bergou, Max Wardetzky, Stephen Robinson, Basile Audoly, and Eitan
  Grinspun.
\newblock Discrete elastic rods.
\newblock In {\em ACM SIGGRAPH 2008 papers}, number~63, pages 1--12. 2008.

\bibitem{bobenko1996discrete}
Alexander~I. Bobenko and Ulrich Pinkall.
\newblock Discrete isothermic surfaces.
\newblock {\em Journal für die reine und angewandte Mathematik},
  1996(475):187--208, 1996.

\bibitem{bobenkoPinkallKNets}
Alexander~I. Bobenko and Ulrich Pinkall.
\newblock {Discrete surfaces with constant negative Gaussian curvature and the
  Hirota equation}.
\newblock {\em Journal of Differential Geometry}, 43(3):527 -- 611, 1996.

\bibitem{discretizationOfSurfacesIntegrable}
Alexander~I. Bobenko and Ulrich Pinkall.
\newblock Discretization of surfaces and integrable systems.
\newblock {\em Discrete integrable geometry and physics}, 16:3--58, 1999.

\bibitem{discreteAffineSpheres}
Alexander~I. Bobenko and Wolfgang~K. Schief.
\newblock Affine spheres: discretization via duality relations.
\newblock {\em Experimental Mathematics}, 8(3):261--280, 1999.

\bibitem{lagrangeTop}
Alexander~I. Bobenko and Yuri~B. Suris.
\newblock Discrete time lagrangian mechanics on lie groups, with an application
  to the lagrange top.
\newblock {\em Communications in mathematical physics}, 204(1):147--188, 1999.

\bibitem{isothermicInSphereGeometriesMoutard}
Alexander~I. Bobenko and Yuri~B. Suris.
\newblock Isothermic surfaces in sphere geometries as moutard nets.
\newblock {\em Proceedings of the Royal Society A: Mathematical, Physical and
  Engineering Sciences}, 463(2088):3171--3193, 2007.

\bibitem{bobenko2008discrete}
Alexander~I. Bobenko and Yuri~B. Suris.
\newblock {\em Discrete differential geometry: Integrable structure},
  volume~98.
\newblock American Mathematical Soc., 2008.

\bibitem{bor2023bicycling}
Gil Bor, Connor Jackman, and Serge Tabachnikov.
\newblock Bicycling geodesics are kirchhoff rods.
\newblock {\em Nonlinearity}, 36(7):3572, 2023.

\bibitem{burstall2015discrete}
Francis Burstall, Udo Hertrich-Jeromin, Wayne Rossman, and Susanna Santos.
\newblock Discrete special isothermic surfaces.
\newblock {\em Geometriae Dedicata}, 174(1):1--11, 2015.

\bibitem{calini2013integrable}
Annalisa Calini, Thomas Ivey, and Gloria Mar{\'\i}~Beffa.
\newblock Integrable flows for starlike curves in centroaffine space.
\newblock {\em SIGMA. Symmetry, Integrability and Geometry: Methods and
  Applications}, 9:022, 2013.

\bibitem{chern2018commuting}
Albert Chern, Felix Knöppel, Franz Pedit, and Ulrich Pinkall.
\newblock {\em Commuting Hamiltonian Flows of Curves in Real Space Forms}, page
  291–328.
\newblock London Mathematical Society Lecture Note Series. Cambridge University
  Press, 2020.

\bibitem{cho2023periodic}
Joseph Cho, Katrin Leschke, and Yuta Ogata.
\newblock Periodic discrete darboux transforms.
\newblock {\em Differential Geometry and its Applications}, 91:102065, 2023.

\bibitem{doliwa2007generalized}
Adam Doliwa.
\newblock Generalized isothermic lattices.
\newblock {\em Journal of Physics A: Mathematical and Theoretical},
  40(42):12539, 2007.

\bibitem{dorfmeister1998weierstrass}
Josef Dorfmeister, Franz Pedit, and Hongyong Wu.
\newblock Weierstrass type representation of harmonic maps into symmetric
  spaces.
\newblock {\em Communications in analysis and geometry}, 6(4):633--668, 1998.

\bibitem{gordon1965zeros}
Basil Gordon and Theodore~Samuel Motzkin.
\newblock On the zeros of polynomials over division rings.
\newblock {\em Transactions of the American Mathematical Society},
  116:218--226, 1965.

\bibitem{hegedus2013factorization}
G{\'a}bor Heged{\"u}s, Josef Schicho, and Hans-Peter Schr{\"o}cker.
\newblock Factorization of rational curves in the study quadric.
\newblock {\em Mechanism and Machine Theory}, 69:142--152, 2013.

\bibitem{hertrich2000transformations}
Udo Hertrich-Jeromin.
\newblock Transformations of discrete isothermic nets and discrete cmc-1
  surfaces in hyperbolic space.
\newblock {\em manuscripta mathematica}, 102(4):465--486, 2000.

\bibitem{discreteDPW}
Tim Hoffmann.
\newblock Discrete cmc surfaces and discrete holomorphic maps.
\newblock {\em Discrete integrable geometry and physics}, pages 97--112, 1999.

\bibitem{hoffmannSmokeRingFlow}
Tim Hoffmann.
\newblock Discrete hashimoto surfaces and a doubly discrete smoke-ring flow.
\newblock In {\em Discrete differential geometry}, pages 95--115. Springer,
  2008.

\bibitem{cKnets}
Tim Hoffmann and Andrew~O. Sageman-Furnas.
\newblock A $ 2\times 2$ lax representation, associated family, and
  b{\"a}cklund transformation for circular k-nets.
\newblock {\em Discrete \& Computational Geometry}, 56(2):472--501, 2016.

\bibitem{hoffmann16}
Tim Hoffmann, Andrew~O. Sageman-Furnas, and Max Wardetzky.
\newblock {A Discrete Parametrized Surface Theory in {$\R^3$}}.
\newblock {\em International Mathematics Research Notices},
  2017(14):4217--4258, 2016.

\bibitem{izosimov2023recutting}
Anton Izosimov.
\newblock Polygon recutting as a cluster integrable system.
\newblock {\em Selecta Mathematica}, 29(2):21, 2023.

\bibitem{langer1991poisson}
Joel Langer and Ron Perline.
\newblock Poisson geometry of the filament equation.
\newblock {\em Journal of Nonlinear Science}, 1:71--93, 1991.

\bibitem{li2019factorization}
Zijia Li, Daniel~F. Scharler, and Hans-Peter Schr{\"o}cker.
\newblock Factorization results for left polynomials in some associative real
  algebras: state of the art, applications, and open questions.
\newblock {\em Journal of Computational and Applied Mathematics}, 349:508--522,
  2019.

\bibitem{li2023geometric}
Zijia Li, Hans-Peter Schr{\"o}cker, and Johannes Siegele.
\newblock A geometric algorithm for the factorization of spinor polynomials.
\newblock {\em arXiv:2310.19325}, 2023.

\bibitem{nijhoff1997some}
Frank Nijhoff.
\newblock On some “schwarzian” equations and their discrete analogues.
\newblock In {\em Algebraic aspects of integrable systems: in memory of Irene
  Dorfman}, pages 237--260. Springer, 1997.

\bibitem{niven1941equations}
Ivan Niven.
\newblock Equations in quaternions.
\newblock {\em The American Mathematical Monthly}, 48(10):654--661, 1941.

\bibitem{ogata2017construction}
Yuta Ogata and Masashi Yasumoto.
\newblock Construction of discrete constant mean curvature surfaces in
  riemannian spaceforms and applications.
\newblock {\em Differential Geometry and its Applications}, 54:264--281, 2017.

\bibitem{pinkallSmokeRingFlow}
Ulrich Pinkall, Boris Springborn, and Steffen Wei{\ss}mann.
\newblock A new doubly discrete analogue of smoke ring flow and the real time
  simulation of fluid flow.
\newblock {\em Journal of Physics A: Mathematical and Theoretical},
  40(42):12563, 2007.

\bibitem{sauer1950parallelogrammgitter}
Robert Sauer.
\newblock Parallelogrammgitter als modelle pseudosph{\"a}rischer fl{\"a}chen.
\newblock {\em Mathematische Zeitschrift}, 52(1):611--622, 1950.

\bibitem{scharler2021algorithm}
Daniel~F. Scharler and Hans-Peter Schr{\"o}cker.
\newblock An algorithm for the factorization of split quaternion polynomials.
\newblock {\em Advances in Applied Clifford Algebras}, 31:1--23, 2021.

\bibitem{scharler2020quadratic}
Daniel~F. Scharler, Johannes Siegele, and Hans-Peter Schr{\"o}cker.
\newblock Quadratic split quaternion polynomials: factorization and geometry.
\newblock {\em Advances in applied Clifford algebras}, 30:1--23, 2020.

\bibitem{schief2001isothermic}
Wolfgang~K. Schief.
\newblock Isothermic surfaces in spaces of arbitrary dimension: integrability,
  discretization, and b{\"a}cklund transformations—a discrete calapso
  equation.
\newblock {\em Studies in Applied Mathematics}, 106(1):85--137, 2001.

\bibitem{schief2007chebyshev}
Wolfgang~K. Schief.
\newblock Discrete chebyshev nets and a universal permutability theorem.
\newblock {\em Journal of Physics A: Mathematical and Theoretical},
  40(18):4775, 2007.

\bibitem{mythesis}
Jannik Steinmeier.
\newblock {On skew parallelogram nets, PhD thesis in preparation}.
\newblock 2024.

\bibitem{sym2005soliton}
Antoni Sym.
\newblock Soliton surfaces and their applications (soliton geometry from
  spectral problems).
\newblock In {\em Geometric Aspects of the Einstein Equations and Integrable
  Systems: Proceedings of the Sixth Scheveningen Conference, Scheveningen, The
  Netherlands, August 26--31, 1984}, pages 154--231. Springer, 2005.

\bibitem{tabachnikov2012discrete}
Serge Tabachnikov and Emmanuel Tsukerman.
\newblock On the discrete bicycle transformation.
\newblock {\em Publ. Mat. Urug. 14}, pages 201--219, 2013.

\bibitem{wunderlich1951differenzengeometrie}
Walter Wunderlich.
\newblock Zur differenzengeometrie der fl{\"a}chen konstanter negativer
  kr{\"u}mmung.
\newblock {\em Osterreich. Akad. Wiss. Math.-Natur. Kl. Sitzungsher. II},
  160:39--77, 1951.

\end{thebibliography}

\end{document}